\documentclass[11pt,oneside, reqno]{amsart}
\usepackage[utf8]{inputenc}
\usepackage{amsmath}
\usepackage{amsthm,ifthen, amsfonts, amssymb,
srcltx, amsopn, color, enumerate}
 \usepackage{mathabx}
\usepackage{amssymb}
\usepackage{bbm}
\usepackage[mathscr]{euscript}
\usepackage{enumerate}
\usepackage{graphicx}
\usepackage{tikz}
\usepackage{tikz-cd}
\usepackage{color}
\usepackage{multicol}
\usepackage{mathrsfs}
\usepackage{hyperref}
\usepackage[normalem]{ulem}
\usepackage{float}
\usepackage{subcaption}
\usepackage{mathabx}
\usepackage{hieroglf}
\usepackage[T1]{fontenc}
\usepackage[cmtip,arrow]{xy}
\usepackage{pb-diagram, pb-xy}
\usepackage{overpic}
\usepackage{verbatim}
\usepackage{thmtools}
\renewcommand{\widehat}{\hat}

\newcommand{\showcommentsbox}{yes}

\newsavebox{\commentbox}
\newenvironment{com}%
{\ifthenelse{\equal{\showcommentsbox}{yes}}%
{\footnotemark
        \begin{lrbox}{\commentbox}
        \begin{minipage}[t]{1.25in}\raggedright\sffamily\tiny
        \footnotemark[\arabic{footnote}]}
{\begin{lrbox}{\commentbox}}}%
{\ifthenelse{\equal{\showcommentsbox}{yes}}%
{\end{minipage}\end{lrbox}\marginpar{\usebox{\commentbox}}}
{\end{lrbox}}}

\definecolor{Green}{RGB}{30, 150, 30}

\usetikzlibrary{arrows}
\usepackage{tikz-cd} 

\newtheorem{thm}{Theorem}[section]
\newtheorem{prop}[thm]{Proposition}
\newtheorem{claim}{Claim}
\newtheorem{lem}[thm]{Lemma}
\newtheorem{lemma}[thm]{Lemma}

\newtheorem{cor}[thm]{Corollary}

\newcounter{intronum}
\newtheorem{conj}[intronum]{Conjecture}
\newtheorem{thmi}{Theorem}
\newtheorem{cori}[thmi]{Corollary}

\newtheorem{problem}[intronum]{Problem}

\theoremstyle{definition}
\newtheorem{defn}[thm]{Definition}
\theoremstyle{definition}

\theoremstyle{definition}

\theoremstyle{definition}
\newtheorem{remark}[thm]{Remark}
\theoremstyle{definition}

\theoremstyle{definition}

\newtheorem{exmp}{Example}
\newtheorem{assume}{Standing Assumption}
\newtheorem{convention}{Convention}

\newcommand*\image{\operatorname{Im}}

\newcommand*\diam{\operatorname{diam}}
\newcommand*\rank{\operatorname{rank}}
\newcommand*\nest{\sqsubseteq}
\newcommand*\propnest{\sqsubsetneq}

\DeclareMathOperator{\hull}{hull}
\newcommand{\dist}{\textup{\textsf{d}}}
\newcommand*\gate{\mathfrak{g}}
\newcommand*\fgate{\mathfrak{f}}
\newcommand*{\orthgate}{\mathfrak e}

\newcommand{\tsh}[1]{\left[#1\right]}
\newcommand{\Tsh}[2]{\tsh{#2}_{#1}}

\newcommand{\orth}{\bot}
\newcommand{\transverse}{\pitchfork}
\newcommand{\neb}{\mathcal N}

\newcommand{\ignore}[2]{\Tsh{#2}{#1}}

\newcommand{\stab}{\operatorname{Stab}}

\newcommand{\growth}{\mathrm{Gr}}
\newcommand{\relevant}{\mathrm{Rel}}

\newcommand{\factored}[1]{\widehat{#1}}
\newcommand{\naturals}{\mathbb N}
\newcommand{\reals}{\mathbb R}
\newcommand{\integers}{\mathbb Z}

\newcommand{\boldmu}{\boldsymbol{\mu}}

\newcommand{\hatcone}{\mathbf{\widehat{X}}}
\newcommand{\tcone}{\mathbf{\widehat{T}}}

\newcommand{\partition}{\mathrm{Part}}
\newcommand{\orthnum}{\#_{\orth}}

\newcommand{\fcone}{\mathbf {\hat f}}

\newcommand{\bC}{\mathbf C}
\newcommand{\bp}{\mathbf p}
\newcommand{\bo}{\mathbf o}
\newcommand{\cdist}{\mathbf{\hat{d}}}
\newcommand{\ctreedist}{\mathbf d}
\newcommand{\bD}{\mathbf D}

\newcommand{\bU}{\mathbf U}
\newcommand{\bV}{\mathbf V}

\newcommand{\jacobian}[1]{\mathrm{J}_{#1}}
\newcommand{\leb}{\mathrm{Leb}}
\newcommand{\haus}{\mathrm{Haus}}
\newcommand{\boundary}{\partial}

\newcommand{\boldgamma}{\boldsymbol{\gamma}}

\newcommand{\asdim}{\mathrm{asdim}}

\newcommand{\MCG}{\mathcal{MCG}}

\newcommand{\DF}{\tau}

\newcommand{\bushy}{\mathfrak S_{\clubsuit}}

\newenvironment{proofofclaim}[1]{
	\begin{proof}[Proof of Claim~#1] }
	{\end{proof}}

\title{Coarse embeddings of products of trees as quasi-isometry invariants}
\author{Mark Hagen}
\address{School of Mathematics, University of Bristol, Bristol, UK}
\email{markfhagen@posteo.net}
\author{Alessandro Sisto}
	\address{Maxwell Institute and Department of Mathematics, Heriot-Watt University,     Edinburgh, UK}
	\email{a.sisto@hw.ac.uk}

\begin{document}

\maketitle

\begin{abstract}
    We consider the maximal number of factors of a product of bushy trees that can be quasi-isometrically, or even coarsely embedded into various groups of interest, including mapping class groups, Torelli groups, Johnson kernels, surface braid groups, and Bestvina-Brady groups. We use this to quasi-isometrically distinguish groups from the above classes, and also to rule out coarse embeddings between them. All these are applications of general statements about coarse embeddings of products of bushy trees into hierarchically hyperbolic spaces.
\end{abstract}


\section{Introduction}\label{sec:intro}
The maximal rank of quasiflats contained in a given space is a well-studied and powerful quasi-isometry invariant, see e.g. \cite{KKL,BehrstockMinsky,EMR,Bowditch:WP} and many others. This quantity is not monotonic under passing to subgroups, so it is natural to look for an invariant of a similar flavour that is. One such invariant is the maximal dimension of coarsely embedded flats, since (finitely generated) subgroups are coarsely embedded in the ambient group. This is, however, hard to control since, for instance, it can be arbitrarily large for a hyperbolic group (because $\mathbb H^n$ contains coarsely embedded copies of $\mathbb R^{n-1}$ as horospheres), which is an undesirable behaviour. In this paper we propose instead the study of coarsely embedded products of free groups/bushy trees. A product of two bushy trees cannot coarsely embed in a hyperbolic group \cite{coarse_emb_into_hyp}, already suggesting that this is a better-behaved invariant. We compute the maximal number of factors of such a product that coarsely (or, in fact, quasi-isometrically) embeds into various hierarchically hyperbolic groups and subgroups of interest. First, for mapping class groups we show the following, where $\Sigma_{g,p}$ is the closed connected oriented surface of genus $g$ with $p$ punctures.

\begin{restatable}{thmi}{mcg}
\label{thm:mcg_bushy_rank}
Let $g\geq 2, p\geq 0$. Each of the following occurs for a positive integer $n$ if and only if $n\leq \left\lfloor \frac{3g+p-2}{2}\right\rfloor$.
\begin{enumerate}
\item $F_2^n$ coarsely embeds in $\MCG(\Sigma_{g,p})$,
\item $F_2^n$ quasi-isometrically embeds in $\MCG(\Sigma_{g,p})$,\label{item:qie-mcg}
\item $\MCG(\Sigma_{g,p})$ contains a quasi-isometrically embedded subgroup isomorphic to $F_2^n$,
\item $\MCG(\Sigma_{g,p})$ contains a subgroup isomorphic to $F_2^n$.
\end{enumerate}
\end{restatable}

The bound $\left\lfloor \frac{3g+p-2}{2}\right\rfloor$ represents the cardinality of a maximal collection of disjoint complexity-1 subsurfaces of $\Sigma_{g,p}$. This is strictly lower than the quasiflats rank of $\MCG(\Sigma_{g,p})$, which is $3g+p-3$ as computed in \cite{BehrstockMinsky}.

Item \eqref{item:qie-mcg}, on quasi-isometric embeddings, is already new. As mentioned, the main motivation is to study subgroups, especially ones whose geometry is poorly understood, and in the mapping class group context we can prove the following:

\begin{restatable}{thmi}{torellijohnsonthm}
\label{thm:Torelli_Johnson}
 Let $\mathcal I_g$ be the Torelli group of genus $g$, and let $\mathcal J_g$ be the Johnson kernel. For $g\geq 3$, $g-1$ is both the maximal $n$ such that $F_2^n$ coarsely embeds in $\mathcal I_g$, and the maximal $n$ such that $F_2^n$ coarsely embeds in $\mathcal J_g$.
\end{restatable}

Note that the bounds given by Theorem \ref{thm:Torelli_Johnson} are lower than the bounds for the ambient mapping class groups, so additional arguments are required, as discussed below Corollary~\ref{cor:bushy-QR}.

Theorem \ref{thm:Torelli_Johnson} can be used to distinguish quasi-isometry classes of Torelli groups and Johnson kernels. Torelli groups of different genera are distinguished up to quasi-isometry by their (virtual) cohomological dimension (see argument below), and same for Johnson kernels, but we are not aware of any other known obstruction before this paper. For instance, while there are bounds on their asymptotic dimension coming from bounds on the asymptotic dimension of mapping class groups (the first ones from \cite{BBF}, the current best ones from \cite{HHS_III}), these are too coarse to distinguish all possible pairs. Combining our results with virtual cohomological dimension considerations we obtain:

\begin{cori}
    No Torelli group is quasi-isometric to any Johnson kernel.
\end{cori}

\begin{proof}
    By Theorem \ref{thm:Torelli_Johnson}, if $\mathcal I_g$ was quasi-isometric to $\mathcal J_h$, then necessarily $g=h$. However, $\mathcal I_g$ and $\mathcal J_g$ have different cohomological dimensions \cite{dimension_torelli}, so they cannot be quasi-isometric by \cite[Theorem 1.2]{qi_vcd}.
\end{proof}

Other interesting subgroups of mapping class groups include surface braid groups. For those we do not compute the optimal number of factors, but using estimates we are still able to prove the following:

\begin{restatable}{thmi}{braids}
    For all $n\geq 5$, $B_n(\Sigma_2)$ is not quasi-isometric to $B_n(\Sigma_n)$.
\end{restatable}

Both the virtual cohomological dimension and the asymptotic dimension of $B_n(\Sigma_g)$ are equal to $n$ \cite{dim_surface_braid}, so these do not provide information in the setting of the theorem.

We can also compare quasi-isometry classes across the types of subgroups that have been mentioned so far.

\begin{restatable}{cori}{torelli_vs_braids}\label{cori:torelli_vs_braids}
\begin{enumerate}
    \item[]
    \item For all $g\geq 6$, $\mathcal I_g$ is not quasi-isometric to any surface braid group.
    \item  No Johnson kernel $\mathcal J_g$, for $g\geq 4$, is quasi-isometric to any surface braid group $B_n(\Sigma_h)$ with $h\geq 3$.
\end{enumerate}   
\end{restatable}

Our results do not only apply to subgroups of mapping class groups, and for instance we can consider subgroups of right-angled Artin groups such as Bestvina-Brady groups.

\begin{thmi}
\label{thmintro:BB}
    There are pairs of Bestvina-Brady groups with the same finiteness properties, Dehn function, and asymptotic dimension, but different values of $n$ such that $F_2^n$ coarsely embeds.
\end{thmi}

Our core technical theorem is Theorem \ref{thm:main}, which, starting with a coarse embedding of $F_2^n$ into a (suitable) HHS $X$, yields a sequence of ``boxes'' in $F_2^n$ which at the level of asymptotic cones give a bilipschitz $n$-flat. To be more precise, this is a bilipschitz $n$-flat in the asymptotic cone of the \emph{factored space} of $X$ which disregards the hyperbolic spaces of the HHS structure that are quasilines, see discussion below. We do not state the theorem here, instead just pointing out its most relevant consequence, on the number of factors of products that coarsely embed in an HHS. For the definition of standard HHS, which covers most natural examples, is Definition \ref{defn:standard-HHS}.

\begin{restatable}{cori}{bushyQR}
Let $(X,\mathfrak S)$ be a standard HHS.  If $f:F_2^k\to X$ is a coarse embedding, then $k$ is at most the maximal cardinality of a collection $\{U_1,\ldots,U_k\}\subset\mathfrak S$ of orthogonal elements such that each $\mathcal CU_i$ is unbounded and not a quasiline.\label{cor:bushy-QR}
\end{restatable}

For the applications above, an improvement on the corollary is also required, Corollary \ref{cor:incompatible}, where the coarse embedding is constrained to have image contained in a specified subspace, and certain hyperbolic spaces do not count for the orthogonality number.

Our results, in particular Corollaries \ref{cor:bushy-QR} and \ref{cor:incompatible}, can also be used to obstruct coarse embeddings between various groups. Here are some examples where neither the asymptotic dimension nor the virtual cohomological dimension suffice (or at least are not known to).

The simplest application for RAAGs is the following.

\begin{exmp}
    The RAAG on a square does not coarsely embed into the RAAG on a pentagon.  
\end{exmp}

The corresponding statement for quasi-isometric embeddings is a special case of \cite[Corollary D]{qie_raag_rigid} and \cite[Corollary 1.6]{poly_hyp}. It is an interesting problem to explore obstructions for coarse embeddings of RAAGs; this is also considered in \cite{coarse_sep_raag}.

In the next two examples, we use the following virtual cohomological dimension and asymptotic dimension facts: $vcd(\MCG(\Sigma_{g}))=4g-5$ for $g\geq 2$, and $vcd(\MCG(\Sigma_{g,q}))=4g-4+q$ for $q\geq 1,g\geq 1$, and $\MCG(\Sigma_{0,p})=p-3$ by \cite{MCGVCD}, and $asdim(\MCG_{0,p})=p-3$ by \cite[Cor. 5]{BellFujiwara}.

\begin{exmp}
Let $g\geq 2,\ p\geq 4$ and $q\ge 0$ satisfy $3g+q<p\leq 4g-2$.  Then $\MCG(\Sigma_{0,p})$ does not coarsely embed in $\MCG(\Sigma_{g,q})$, but this cannot be determined from virtual cohomological dimension or asymptotic dimension considerations since  $\asdim(\MCG(\Sigma_{g,q}))\geq \asdim(\MCG(\Sigma_{0,p}))$
and $vcd(\MCG(\Sigma_{g,q}))\geq vcd(\MCG(\Sigma_{0,p}))$.
\end{exmp}

\begin{exmp}
    One cannot coarsely embed $\MCG(\Sigma_{g-1,4})$ in $\MCG(\Sigma_g)$, but virtual cohomological dimension does not show this, since $vcd(\MCG(\Sigma_{g-1,4}))\leq vcd(\MCG(\Sigma_g)).$ The best known asymptotic dimension bounds, from \cite{HHS_III}, are too coarse to yield such a result.
\end{exmp}

Note that $\MCG(\Sigma_{g-1,2})$ quasi-isometrically embeds into $\MCG(\Sigma_g)$ (via the stabiliser of a non-separating curve, which is quasi-isometric to a product where one factor is $\MCG(\Sigma_{g-1,2})$). We do not know about the case of 3 punctures.

\subsection{Outline of paper and proofs}  Section \ref{sec:prelim} contains preliminary material on hierarchical hyperbolicity.  In Section \ref{sec:main-theorem}, we state the main technical theorem and deduce all the results in the introduction as consequences.  In Sections \ref{subsec:counting} and \ref{sec:induction}, we prove Theorem \ref{thm:main}.

First of all, let us recall the construction of factored spaces from \cite{HHS_III}. An HHS $X$ has a ``layered'' structure (for mapping class groups, this is given by subsurfaces and their inclusions), with each layer indexing a set of hyperbolic spaces (for mapping class groups, these curve graphs of subsurfaces). There is a method to form a cone-off $\widehat X$ of $X$, which we call a factored space, and which is still an HHS where our favourite elements of the bottom layer of the hierarchy have been discarded (for the mapping class group, we are discarding annuli, and the factored space is quasi-isometric to the pants graph). We will perform this construction for HHSs where all quasilines for the structure are at the bottom, and can therefore be discarded. Morally, quasilines are not big enough to fully account for coarsely embedded products of bushy trees, so we want to discard them.

We compose a given coarse embedding $f:T^k_3\to X$, where $T_3$ is the regular tree of valence 3, with the map $X\to\widehat X$ to the factored space described above, to get a lipschitz map $\hat f:T^k_3\to\widehat X$ that is no longer a coarse embedding.  Still, the point of Section \ref{prop:counting} is to prove Proposition \ref{prop:counting}, which says that $\hat f$ ``behaves like a quasi-isometric embedding'' in most directions. Since we assume that $X$ has bounded geometry, a simple and known counting argument (exploited also for instance in \cite{coarse_emb_into_hyp}) yield that $f$ behaves like a quasi-isometric embedding in most directions, see Lemma \ref{lem:easy-count}, so what we have to prove is that the natural map $X\to \hat X$ does as well.

Towards this, we would like to count points that are linearly far from the basepoint in $X$ but sublinearly far in $\widehat X$. First observe that, given a basepoint in the standard HHS $X$, there are linearly many points that differ significantly from the basepoint in a given quasiline coordinate and nowhere else; see Lemma \ref{lem:linear} for the refinement of this fact that we will need. More generally, Lemma \ref{lem:sequence} says that one can count the points in question by considering paths $\alpha_0\beta_1\cdots\alpha_{n-1}\beta_n\alpha_n$, where the  $\sum_i|\alpha_i|$ of lengths $|\alpha_i|$ is sublinear, and each $\beta_i$ joins points $x_i,x_{i+1}$ with a similar property to the above observation.  This reduces the question to counting the possible tuples of positive integers that could appear as lengths $|\alpha_i|,|\beta_i|$ for such a path, i.e. to counting integer partitions.  We rely on a nontrivial estimate from the literature, see Lemma \ref{lem:partition-counting} and its proof, which is just sharp enough for our purposes. 

Finally, in Section \ref{sec:induction} we complete the proof of Theorem \ref{thm:main} on coarse embeddings of $T_3^k$, using induction on $k$. The main tools in this section are asymptotic cones, geometric analysis, and basic algebraic topology of manifolds.

Passing to asymptotic cones and keeping Proposition \ref{prop:counting} into account, we obtain a lipschitz map from a product of trees into an asymptotic cone of $\widehat X$ which is ``bilipschitz in many directions''. Geometric analysis is mostly used to find points where certain lipschitz maps, between subspaces of asymptotic cones bilipschitz equivalent to subspaces of $\mathbb R^n$, have full rank Jacobian. We state and prove self-contained results in Section \ref{subsec:gmt-facts} to this end. Changing asymptotic cones, we can upgrade from full rank Jacobians to bilipschitz maps. Identifying suitable subspaces of the relevant products of trees where we can apply the geometric analysis arguments is one of the main challenges in the arguments. Some of the arguments in Section \ref{sec:induction} are taken from \cite[Section 13]{HHS_I}, where quasiflats are studied, but we use substantially more geometric analysis and other tools here because we have to identify suitable subspaces along the way.

\subsection{Problems}
There are four main questions arising from our work that we want to emphasise. The first is about handlebody groups.

\begin{conj}
    The maximal integer $n$ such that $F_2^n$ quasi-isometrically embeds into the handlebody group of genus $g$ is $n=g-1$.
\end{conj}

In order to prove the conjecture, because of Corollary \ref{cor:incompatible}, it suffices to show that the collection of all one-holed tori is incompatible with the handlebody group, as defined in Definition \ref{defn:incompatible}. This would be the same strategy as for Torelli groups.

Secondly, while we obtain bounds in the context of surface braid groups, we do not determine optimal ones. It is then natural to pose the following:

\begin{problem}
    Given a surface braid group, find bounds, ideally optimal ones, on the number of factors of products of free groups that coarsely embed. Use this to distinguish quasi-isometry classes of surface braid groups.
\end{problem}

A more open-ended problem regards groups with proper non-cocompact cubulations. Such groups come with a coarse embedding in the corresponding CAT(0) cube complex, and many CAT(0) cube complexes are HHSs, including convex subcomplexes of universal covers of Salvetti complexes \cite{HHS_I, HagenSusse}. Therefore, it should be possible to study coarse and quasi-isometric embeddings of $F_2^n$ into such groups relying on the results in this paper. In particular, the case of Coxeter groups is of great interest; these are indeed cubulated and the corresponding CAT(0) cube complexes are HHS, see in \cite{NibloReeves:coxeter,HaglundWise:coxeter}.

\begin{problem}
    Study coarse and quasi-isometric embeddings of product of free groups into Coxeter groups.
\end{problem}

Finally, as mentioned above, we want to advertise the following.

\begin{problem}
    Study coarse embeddings between RAAGs.
\end{problem}

\subsection*{Acknowledgments} We would like to thank Sebastian Hensel for very useful discussions and insights about subgroups of mapping class groups. Also, we would like to thank Giorgio Mangioni for useful feedback.

\section{Preliminaries}\label{sec:prelim}

\subsection{Standard HHSs}\label{subsec:standard-bushy-factored}
We mostly follow the treatment of HHSs from \cite{HHS_II,HHS_III} and refer the reader to \cite[Part 3]{CRHK} for additional details on HHS background.

\begin{defn}\label{defn:gromov-product}
Let $(M,d)$ be a metric space and let $m\in M$.  As usual, the \emph{Gromov product} $(-,-)_m:M^2\to \reals_{\geq 0}$ is given by
$$(x,y)_m=\frac{1}{2}\left(d(m,x)+d(m,y)-d(x,y)\right)$$
for $x,y\in M$.
\end{defn}

All of the (Gromov-) hyperbolic spaces in this paper are hyperbolic geodesic spaces.

\begin{defn}[Bushy hyperbolic space]\label{defn:bushy}
An $E$--hyperbolic geodesic space $X$ is \emph{$E$--bushy} (or just \emph{bushy} when $E$ is understood) if, for all $p\in X$, there exist $(1,E)$--quasigeodesic rays $\alpha_0,\alpha_1,\alpha_2$ in $X$ such that $\alpha_i(0)=p$ and, letting $a_i\in\partial X$ be the endpoint of $\alpha_i$, we have $(a_i,a_j)_p\leq E$ whenever $i\neq j$.
\end{defn}

\begin{defn}[Standard HHS, factored space, factored map]\label{defn:standard-HHS}
An HHS $(X,\mathfrak S)$ is \emph{standard} if there is a constant $E$ such that all of the following hold:
\begin{enumerate}
    \item For all $U\in\mathfrak S$, either $\mathcal CU$ is unbounded or $\diam(\mathcal CU)\leq E$.\label{item:standard-unbounded}

    \item For all $U\in\mathfrak S$ such that $\mathcal CU$ is unbounded and $U$ is not $\nest$--minimal, $\mathcal CU$ is an $E$--bushy $E$--hyperbolic space.\label{item:standard-bushy}

    \item The underlying metric space $(X,\dist)$ has bounded geometry.\label{item:standard-bounded-geometry}

    \item For each $U\in\mathfrak S$ such that $\mathcal CU$ is unbounded but not $E$--bushy, the space $\mathcal CU$ is an $(E,E)$--quasiline or an $(E,E)$--quasi-ray.\label{item:standard-quasi-line}

    \item $\pi_U:X\to\mathcal CU$ is $E$--coarsely surjective for all $U\in\mathfrak S$.\label{item:standard-surjective}
\end{enumerate}
Given a standard HHS $(X,\mathfrak S)$, let $\mathfrak S_{ql}$ be the set of $U\in\mathfrak S$ such that $\mathcal CU$ is unbounded but not $E$--bushy. Since $\mathfrak S_{ql}$ consists of $\nest$--minimal elements, from \cite[Prop. 2.4]{HHS_III} we have a factored space $(\widehat X,\mathfrak S-\mathfrak S_{ql})$.  Let $q:(X,\dist)\to(\widehat X,\hat\dist)$ be the set-theoretic identity, which is lipschitz, where $\hat\dist$ is the factored metric from \cite[Prop. 2.2]{HHS_III}.  Given any map $f:A\to X$, let $\hat f=q\circ f$.
\end{defn}

\begin{lemma}
\label{lem:hhg_std}
Let $(G,\mathfrak S)$ be an HHG such that $\stab_G(U)$ acts coboundedly on $\mathcal C U$ for all $U\in\mathfrak S$. Then $G$ has an HHG structure $(G,\mathfrak S')$ which is standard. More precisely, $\mathfrak S'$ is obtained from $\mathfrak S$ by removing a collection of bounded domains.
\end{lemma}

\begin{proof}
    Property \ref{item:standard-unbounded} follows from the fact that there are finitely many orbits of domains.  We can also choose $E$ so that $\stab_G(U)$ acts on $\mathcal CU$ $E$--coboundedly.  Property \eqref{item:standard-bounded-geometry} follows since the underlying metric space is just $G$ with a proper left-invariant metric.  Property \eqref{item:standard-surjective} can be assume to hold by \cite[Rem. 1.3]{HHS_II}.  
    
    By \cite[Prop. 2.4]{HHS_III}, we can assume that for all $U\in\mathfrak S$, there exists $V\nest U$ such that $\mathcal CV$ is unbounded, and that if $U$ is not $\nest$--minimal, such $V$ can be chosen with $V\propnest U$.  Let $U$ be such that $\mathcal CU$ is unbounded and not $\nest$--minimal and let $V\propnest U$ have $\mathcal CV$ unbounded.  Then Lemma \ref{lem:bush} implies that $\mathcal CU$ is bushy, giving item \eqref{item:standard-bushy}.

    Finally, if $\mathcal CU$ is unbounded but $U$ is not bushy, then coboundedness implies that $\mathcal CU$ is a quasiline, as required by property \eqref{item:standard-quasi-line}.     
\end{proof}

\begin{lemma}\label{lem:bush}
    Let $(G,\mathfrak S)$ be an HHG such that $\stab_G(U)$ acts $E$--coboundedly on $\mathcal C U$ for all $U\in\mathfrak S$. Whenever $U,V\in\mathfrak S$ are such that $V\propnest U$ and $\mathcal C U,\mathcal C V$ are unbounded, then $\mathcal C U$ is bushy.
\end{lemma}

\begin{proof}
Let $p\in\mathcal CU$.  Using coboundedness, choose $V$ so that $\dist_U(p,\rho^V_U)\leq E$.  Choose $x\in G$ so that $\dist_U(x,\rho^V_U)>100E$ and $x$ lies on some $(1,E)$--quasigeodesic ray $\gamma$ in $\mathcal CU$ starting at $p$.  Since $\stab_G(V)$ acts on $\mathcal CV$ coboundedly and $\mathcal CV$ is unbounded, there exists $a,b\in \stab_G(V)$ such that $\dist_V(gx,hx)>100E$ for distinct $g,h\in\{1,a,b\}$, so by consistency and bounded geodesic image, geodesics in $\mathcal CU$ from $\pi_U(ax)$ to $\pi_U(x)$ must pass $E$--close to $\rho^V_U$, so any two of $\gamma$ and $a\gamma$ and $b\gamma$ are, after modifying uniformly bounded initial segments so that they all start at $p$ (instead of $E$--close), rays witnessing bushyness of $\mathcal CU$.
\end{proof}

\subsection{Product regions, hierarchy intervals, distance formula sum}\label{subsec:products-paths}
Fix an HHS $(X,\mathfrak S)$.  Let $\mu:X^3\to X$ be the coarse median operator from \cite[Sec. 7]{HHS_II}.

\begin{defn}\label{defn:HQC}
A subset $Y\subset X$ is \emph{$K$--quasimedian quasiconvex} if $\dist(\mu(y,y',x),Y)\leq K$ whenever
 $y,y'\in Y$ and $x\in X$.  
\end{defn}

By \cite[Prop. 5.11]{RussellSprianoTran}, quasimedian quasiconvexity is equivalent to the property of \emph{hierarchical quasiconvexity} from \cite[Sec. 5]{HHS_II}.  If $Y\subseteq X$ is $K$--quasimedian quasiconvex, then $\gate_Y:X\to Y$ is the \emph{gate map} from \cite[Sec. 5]{HHS_II}; we will sometimes use the term \emph{(coarse) gate map} to avoid confusion with the conceptually similar but formally different notion of gate maps in median spaces (which we will also use).

\begin{remark}\label{rem:product-region-facts}
We summarise facts we will need about standard product regions in $X$ as follows (see \cite[Sec. 5]{HHS_II} and \cite[Sec. 17]{CRHK}).  We can and shall subsume various constants into one constant $E$.
\begin{enumerate}
    \item For each $U\in\mathfrak S$, the \emph{standard product region} $P_U$ is the set of $x\in X$ such that $\dist_V(x,\rho^U_V)\leq E$ whenever $U\propnest V$ or $U\transverse V$.

    \item For all $W\nest U$ or $W\orth U$, the restriction of $\pi_W$ to $P_U$ is $E$--coarsely surjective.

    \item For each $p\in P_U$, let $F_U^p$ be the set of $x\in P_U$ such that $\dist_W(x,p)\leq E$ for all $W\orth U$, and let $E_U^p$ be the set of $y\in P_U$ such that $\dist_V(y,p)\leq E$ for all $V\nest U$.  Then $P_U$, $F_U^p$, and $E_U^p$ are $E$--quasimedian quasiconvex subsets of $X$.  

    \item For $V\nest U$, the restriction of $\pi_V$  to $F_U^p$ is $E$--coarsely surjective, and the same is true for the restriction of $\pi_W$ to $E_U^p$ whenever $W\orth U$.

    \item The inclusions $F_U^p,E_U^p$ extend to a quasi-isometric embedding $F_U^p\times E_U^p\to X$ whose image is $P_U$ and which is $E$--quasimedian for the product coarse median on the domain.  

    \item If $U_1,\ldots,U_n\in\mathfrak S$ are pairwise-orthogonal, then $\bigcap_{i=1}^nP_{U_i}\neq\emptyset$, this intersection is uniformly hierarchically quasiconvex, and for all $p\in\bigcap_{i=1}^nP_{U_i}$ and $j\leq n$, we have $F_{U_j}^p,E_{U_j}^p\subseteq \bigcap_{i=1}^nP_{U_i}$, and, moreover, the image of $F_{U_j}^p\times E_{U_j}^p\to X$ is contained in $\bigcap_{i=1}^nP_{U_i}$.  
    
    Finally, the maps $F_{U_i}^p\to X$ extend to a uniformly quasimedian, uniform quasi-isometric embedding $\prod_iF^p_{U_i}\to X$ whose image is uniformly hierarchically quasiconvex and contained in $\bigcap_iP_{U_i}$.
\end{enumerate}
We sometimes write $F_U$ (resp. $E_U$) for $F_U^p$ (resp. $E_U^p$) when the particular choice of coarse parallel copy is understood or unimportant.  We let $\fgate_U^p:X\to F_U^p$ be the gate map and $\orthgate_U^p:X\to E_U^p$.  After uniformly enlarging $E$, we can assume that all of these maps are $(E,E)$--coarsely lipschitz.
\end{remark}

\begin{defn}[Hierarchy path]\label{defn:hierarchy-path}
Given a constant $D$, a $D$--\emph{hierarchy path} is a $(D,D)$--quasigeodesic $\gamma:[0,L]\to X$ such that $\dist(\mu(\gamma(i),\gamma(j),\gamma(k)),\gamma(j))\leq D$ whenever $0\leq i\leq j\leq k\leq L$.

There exists a constant $D$, depending only on $E$, such that any two points in $X$ are joined by a $D$--hierarchy path, and if $Y\subseteq X$ is $E$--quasimedian quasiconvex, then any two points in $Y$ are joined by a $D$--hierarchy path in $Y$.
\end{defn}

By \cite[Prop. 1.37]{HHS:quasiflats}, Definition \ref{defn:hierarchy-path} is equivalent (up to uniform change of constants) to the notion of a hierarchy path in \cite[Sec. 4]{HHS_II}.

\begin{defn}[Hierarchy interval]\label{defn:coarse-interval}
By \cite[Lem. 6.2]{HHS_II}, there exists $\theta$ such that the following holds.  Given $x,y\in X$, the set $\hull(x,y)$ of $z\in X$ such that $\dist(\mu(x,y,z),z)\leq \theta$ is an $E$--quasimedian quasiconvex subset that we will call the \emph{hierarchy interval} between $x,y$.
\end{defn}

We now recall the \emph{distance formula} and set related notation.

\begin{defn}[$K$--relevant domains]\label{defn:relevant}
Given $K\geq 0$ and $x,y\in X$, let $$\relevant_K(x,y)=\{U\in\mathfrak S:\dist_U(x,y)\geq M\}.$$
\end{defn}

\begin{defn}[Threshold notation]\label{defn:threshold}
Given real numbers $A,B$, let $\ignore{A}{B}$ if $A\geq B$ and $0$ otherwise.
\end{defn}

\begin{defn}[Distance formula sum]\label{defn:DF-sum}
Let $\lambda:[0,\infty)\to[0,\infty)$ be an affine function and let $\tau\geq 0$.  For $V\in\mathfrak S$, let $\mathfrak S_V=\{U\in\mathfrak S:U\nest V\}$.  For $x,y\in X$, and $V\in\mathfrak S$, set
$$\sigma^{\tau}_{V,\lambda}=\sum_{U\in\mathfrak S_V}\ignore{\lambda(\dist_U(x,y))}{\tau}=\sum_{U\in\mathfrak S_V\cap \relevant_\tau(x,y)}\lambda(\dist_U(x,y)).$$
In contexts where we have fixed $\tau$ and will not vary it, we omit it and write $\sigma^\lambda_{V,\lambda}$.  If $\lambda$ is the identity, we omit it, and if $V$ is the unique $\nest$--maximal element of $\mathfrak S$, we omit it.
\end{defn}

The strengthened distance formula, which is \cite[Thm. 2.9]{CHHS_I}, says that for all $\lambda$, there exists $s_\lambda$, depending only on $\lambda$ and the HHS parameters, such that for all $s\geq s_\lambda$, there exists $\DF(s,\lambda)$ such that for all $V\in\mathfrak S$, all $p\in P_V$, and all $x,y\in F^p_V$, 
$$\frac{\dist(x,y)}{\DF(s,\lambda)}-\DF(s,\lambda)\leq\sigma_V^{s,\lambda}(x,y)\leq \DF(s,\lambda)\cdot \dist(x,y)+\DF(s,\lambda).$$

In particular, the above holds for all $x,y\in X$ when $\sigma_V^{s,\lambda}$ is replaced by $\sigma^{s,\lambda}$.  When $\lambda$ is the identity, we write $\DF(s):=\DF(s,Id)$ to simplify notation, so that $s_{Id}$ and $\DF(s)$ are the constants provided by the original distance formula, \cite[Thm. 4.5]{HHS_II}.

Later, when we work with both $(X,\mathfrak S)$ and a fixed factored space $(\widehat X,\mathfrak S-\mathfrak S_{ql})$, we can and shall assume that the same $s_\lambda$ and constants $\DF(s,\lambda)$ work for both $\dist$ and $\hat\dist$, just by taking the maximum of the constants needed for the two different metrics.

The following observation morally says that as one moves along a hierarchy path from $x$ to $y$, one does not introduce surprising new distance formula terms:

\begin{lem}\label{lem:initial-M}
Let $C\geq 0$ and let $M_1=4\theta+1$.  Let $x,y\in X$ and suppose that $x',y'\in X$ have the property that $\pi_V(x'),\pi_V(y')$ are contained in the $(\theta+C)$--neighbourhood of any $\mathcal CV$--geodesic from $\pi_V(x)$ to $\pi_V(y)$ for all $V\in\mathfrak S$.  Then $\relevant_{2M}(x',y')\subseteq \relevant_M(x,y)$ for all $M\geq M_1+4C$.  
\end{lem}

\begin{proof}
Fix $C\ge0$ and let $x,y,x',y'\in X$ be as in the statement.  Then for all $V\in\mathfrak S$, letting $\gamma_V$ be a geodesic in $\mathcal CV$ from $\pi_V(x)$ to $\pi_V(y)$, we have $\dist_V(y',\gamma_V),\dist_V(x',\gamma_V)\leq \theta+C$.  Hence, up to relabelling,
$\dist_V(x,x')+\dist_V(x',y')+\dist_V(y',y)\leq \dist_V(x,y)+4(\theta+C)$.  Thus, if $M$ is such that $\dist_V(x',y')\geq 2M$ but $\dist_V(x,y)\leq M$, then $M\leq 4(\theta+C)$.  Hence $M_1=4\theta+1$ has the required property.
\end{proof}

\begin{remark}[Factored notation]\label{rem:factored-stuff}
If $(X,\mathfrak S)$ is standard and $(\widehat X,\mathfrak S-\mathfrak S_{ql})$ is the factored space from Definition \ref{defn:standard-HHS}, we let $\hat\mu$ be the coarse median on $\widehat X$.  Up to enlarging $E$ uniformly, we can and shall assume that the above discussion of standard product regions, etc. all holds for the HHS $(\widehat X,\mathfrak S-\mathfrak S_{ql})$ with the same constant $E$, and $\hat \mu=q\circ\mu$.  We use $\widehat F_U^p$ and $\widehat E_U^p$ to mean the coarse factors of the $U$--standard product region in $\widehat X$, etc.  We use $\hat \sigma_V$ for distance formula sums in the HHS $(\widehat X,\mathfrak S-\mathfrak S_{ql})$.
\end{remark}

\subsection{Standard flats and gates}\label{subsec:hedge-stuff}
Fix a non-principal ultrafilter $\omega$ on $\naturals$.  Given sequences $(a_n)_n,(b_n)_n$ in $\reals$, we write
\begin{itemize}
    \item $a_n=O(b_n)$ to mean there exists $M<\infty$ with $|a_n|\leq M|b_n|$ for $\omega$--a.e. $n$;

    \item $a_n=o(b_n)$, or $(a_n)_n\ll (b_n)_n$, to mean $\omega-\lim_n \frac{a_n}{b_n}=0$.

\end{itemize}

Fix a standard HHS $(X,\mathfrak S)$ and suppose $U_1,\ldots,U_n\in \mathfrak S-\mathfrak S_{ql}$ satisfy $U_i\orth U_j$ for all $1\leq i<j\leq n$.  

\begin{defn}[Hierarchy box, face]\label{defn:hierarchy-box}
Fix $p\in\bigcap_{i=1}^n\widehat P_{U_i}$.  For each $i$, let $\gamma_i\subset \widehat X$ be a hierarchy interval between points $a_i,b_i\in\widehat F_{U_i}^p$.  There exists $D'$, depending on $E$, such that the inclusions $\gamma_i\to\widehat X$ extend to a $D'$--quasimedian, $(D',D')$--quasi-isometric embedding $H:=\prod_{i=1}^n\gamma_i\to \widehat X$ with image in $\neb_{D'}(\bigcap_{i=1}^nP_{U_i})$.  We call $H$ a \emph{hierarchy box} for $\{U_1,\ldots,U_n\}$.  For each $i$, an associated \emph{face} of $H$ is the hierarchy box obtained by restricting the map $H\to \widehat X$ to a subset obtained by setting the $i^{th}$ coordinate to $a_i$ or $b_i$.
\end{defn}

Now fix a sequence $(o_n)_n$ of basepoints in $\widehat X$ and a rescaling sequence $(r_n)_n$.  For each $n$, let $H_n=\prod_{i=1}^k\gamma^n_i$ be a hierarchy box for some pairwise-orthogonal set $\{U_1^n,\ldots,U_k^n\}$.  Let $\mathcal F_n$ be the set of faces of $H_n$. 

\begin{defn}\label{defn:valid-boxes}
The sequence $(H_n)_n$ of hierarchy boxes is \emph{valid} if the following holds. Define
$$s_n:=\min_i \hat\dist(a^n_i,b^n_i),$$
and
$$e_n:=\max_i\sup_{V\propnest U^n_i} \hat\sigma_V(a_i^n,b_i^n).$$
Then we require:
\begin{itemize}
    \item $(s_n)_n\gg (r_n)_n$,
    \item $(e_n)_n\ll (r_n)_n$, and
    \item $\hat\dist(o_n,H_n)=O(r_n)$ and $(\hat\dist(o_n,F_n))_n\gg (r_n)_n$ for any sequence of faces $F_n\in\mathcal F_n$.
\end{itemize}
We say that $(H_n)_n$ is a \emph{small} valid sequence of hierarchy boxes if it satisfies the following modified list of conditions:
\begin{itemize}
    \item $s_n=O(r_n)$.
    \item $(e_n)_n\ll (r_n)_n$.
    \item $\hat\dist(o_n,H_n)=O(r_n)$.
\end{itemize}
In either case, let $\bf H=\omega-\lim_n H_n$ and let $\hatcone=Cone_\omega(\factored X,(r_n),(o_n))$.  The maps $H_n\to \widehat X$ limit to an injective map $\bf H\to \mathbf{\widehat X}$, whose image we also denote $\bf H$.
\end{defn}

Recall from \cite[Prop. 9.1]{Bowditch:coarse_median} that defining $\boldmu:\hatcone^3\to\hatcone$ as $\boldmu((x_n)_n,(y_n)_n,(z_n)_n)=\omega-\lim_n\mu(x_n,y_n,z_n)$ makes $\hatcone$ into a complete connected topological median algebra, and the median rank is bounded by the HHS complexity of $(\widehat X,\mathfrak S-\mathfrak S_{ql})$, as explained in \cite[Prop. 1.35]{HHS:quasiflats} or \cite[Sec. 25.2]{CRHK}.  In particular, if $\bf Y\subseteq\hatcone$ is closed and median-convex, then there is a \emph{(median) gate map} $\gate_{\bf Y}:\hatcone\to \bf Y$ characterised by the property that $\boldmu(\mathbf x,\gate_{\bf Y}(\mathbf x),\mathbf y)=\gate_{\bf Y}(\mathbf x)$ for all $\mathbf x\in\hatcone$ and $\mathbf y\in\bf Y$.

\begin{lem}\label{lem:hedge-gate-support}
Let $(H_n)_n$ be a valid sequence of hierarchy boxes.  Then the (coarse) gate maps $\widehat X\to H_n$ limit to the (median) gate map ${\bf\gate}:{\bf\factored X}\to\bf H$.  Moreover, $\bf H$ is bilipschitz equivalent to $\reals^k$, and ${\bf\gate}$ is lipschitz (with constant depending only on the HHS parameters).

The same holds if $(H_n)_n$ is a \emph{small} valid sequence, except that $\bf H$ is bilipschitz equivalent to $\prod_{i=1}^k[0,L_i]$ where $L_i\in(0,\infty]$ for all $i$.
\end{lem}

\begin{proof}
Since the sequence $(H_n)_n$ is valid, $\bf H$ is a well-defined closed subset of $\hatcone$. Each $H_n$ is uniformly quasimedian quasiconvex in $\widehat X$, by construction, so $\bf H$ is median-convex and the statement about the gate maps follows from, for instance, \cite[Lem. 3.1]{HHS:quasiflats}.

Let $H_n$ be a hierarchy box for $\{U_1^n,\ldots,U_k^n\}$, so that $H_n$ is the image of $\gamma^n:\prod_{i=1}^k\gamma^n_i\to \neb_{D'}(\bigcap_{i=1}^kP_{U^n_i})$, where $\gamma^n$ is a $D'$--quasimedian $(D',D')$--quasi-isometric embedding.  So, $\bf H$ is the image of a $D'$--bilipschitz median map $\boldsymbol{\gamma}:\prod_{i=1}^k\boldsymbol{\gamma}_i\to \mathbf{\widehat X}$, with convex image, where $\boldsymbol{\gamma}_i$ is the rescaled ultralimit of the hierarchy intervals $\gamma^n_i$.  

Fix $i\leq k$ and let $(V_n)_n$ be a sequence in $\mathfrak S-\mathfrak S_{ql}$ such that $V_n\propnest U_i^n$ for $\omega$--a.e. $n$ and $\dist(o_n,p_n)=O(r_n)$ for some sequence $(p_n)_n$ with $p_n\in P_{V_n}$.  The condition on $(e_n)_n$ in Definition \ref{defn:valid-boxes} implies that $\diam(\fgate^{p_n}_{V_n}(H_n))=o(r_n)$, so the image of $\bf H$ under the gate map to the closed convex set $\omega-\lim_n F^{p_n}_{V_n}$ in $\hatcone$ is a single point.  If $(V_n)_n$ is a sequence with $V_n\orth U^n_i$ for $\omega$--a.e. $n$, then the gate map to $\omega-\lim_n F^{p_n}_{V_n}$ sends each parallel copy of $\boldsymbol{\gamma}_i$ to a point, so $\boldsymbol{\gamma}_i$ is closed and convex and lies in the convex set $\omega-\lim_n F_{U^n_i}^{p_n}$.  From \cite[Thm. 28.3]{CRHK}, and the fact that the gate map to nested $\omega-\lim_n F^{p_n}_{V_n}$ is constant on $\boldsymbol{\gamma}_i$, we get that $\boldsymbol{\gamma}_i$ is a real tree, and it is therefore an interval, a line, or a ray, since each $\gamma_i^n$ is a hierarchy interval.  The assumptions on $(s_n)_n$ and distances from faces to $o_n$ in Definition \ref{defn:valid-boxes} rule out bounded intervals and rays.  So $\bf H$ is bilipschitz equivalent to $\reals^k$.

If $(H_n)_n$ is a small valid sequence, then the same argument as above applies to show that $\bf H$ is median convex, and is the image of a bilipschitz median map $\boldgamma:\prod_{i=1}^k\boldgamma_i\to \hatcone$, where each $\boldgamma_i$ is the rescaled ultralimit of the $(\gamma_n^i)_n$, and $\boldgamma_i$ is a nontrivial interval, line, or ray.
\end{proof}

\begin{defn}[Standard flat, standard box, gate]\label{defn:hedge}
 Under the conditions of Lemma \ref{lem:hedge-gate-support}, if $(H_n)_n$ is a valid sequence, we call $\bf H$ a \emph{standard flat} of \emph{rank $k$}.  If $(H_n)_n$ is a small valid sequence, we call $\bf H$ a \emph{standard box}.  In either case, we call ${\bf\gate}$ the \emph{gate} to $\bf H$. Note that a standard box is bilipschitz equivalent to the product of finitely many intervals, rays, and lines, by Lemma \ref{lem:hedge-gate-support}.
\end{defn} 

\begin{remark}\label{rem:hedge-flat-dimension}
Each standard flat or standard box $\bf H$ is, by definition, associated to a sequence of sets $\{U_i^n\}_{i=1}^k$ of orthogonal bushy domains in $\mathfrak S$, where, by Lemma \ref{lem:hedge-gate-support}, $k=\dim \bf H$.  For each $i\leq k$, let $\mathbf U_i=\omega-\lim U_i^n$.  
\end{remark}

\begin{defn}\label{defn:support-standard-flat}
The \emph{support} of the standard flat/box $\mathbf H$ is the set $\{\mathbf U_i\}_{i=1}^k$ from Remark \ref{rem:hedge-flat-dimension}.  If $\mathbf H'$ is another standard flat/box with support $\{\mathbf V_j\}_{j=1}^{k'}$, then we say $\mathbf H,\mathbf H'$ have \emph{orthogonal supports} if $U_i^n\orth V_j^n$ for all $i,j$ and $\omega$--a.e. $n$. The real trees $T\mathbf U_i=\omega-\lim_n (\mathcal CU^n_i,\dist_{U^n_i}/r_n)$ are the \emph{support trees} for $\mathbf H$.  We let $\dist_{\mathbf U_i}=\omega-\lim_n \dist_{U^n_i}/r_n$ and we define $\pi_{\mathbf U_i}:\mathbf{\widehat X}\to T\mathbf U_i$ by $\pi_{U_i}=\omega-\lim_n\pi_{U^n_i}$.\footnote{We warn the reader that the real trees $T\mathbf U_i$ are not the same as the real trees associated to $\mathbf U_i$ in the real cubing structure in \cite{CRHK}, although they are closely related, and the same is true of the maps $\pi_{\mathbf U_i}$.}  
\end{defn}

\section{Main theorem and consequences}\label{sec:main-theorem}
Our main result is Theorem \ref{thm:main}.  In the theorem, $T_3$ refers to the $3$--regular simplicial tree with the usual path metric  $d$. Given $a,b\in T_3$, let $[a,b]$ denote the (image of the) geodesic joining them.  We also use $d$ to refer to the $\ell_1$ metric on $T_3^k$, although the content of the theorem is not affected by replacing the $\ell_1$ metric with any bilipschitz equivalent metric.

\begin{restatable}{thm}{main}
\label{thm:main}
    Let $( X,\mathfrak S)$ be a standard HHS, let $k\geq 0$, and let $f:T_3^k\to  X$ be a coarse embedding. Then there exist
 \begin{itemize}
 \item an asymptotic cone ${\bf \factored X}$ of $\factored X$ and
     \item a sequence $C_n=\Pi_{i=1}^k[a^n_i,b^n_i]\subseteq T_3^k$
\end{itemize}
such that the following holds.  Let ${\bf\hat f}:{\bf C}\to {\bf \factored X}$ be the ultralimit of $\hat f|_{C_n}:C_n\to \widehat X$, where $\mathbf C$ has the metric $\ctreedist=\omega-\lim_n d/r_n$.  Then there exists a standard $k$--flat $\bf F$ in ${\bf \factored X}$, with gate $\bf\mathfrak g$, such that the composition ${\bf \mathfrak g}\circ{\bf \hat f}:{\bf C}\to{\bf F}$ is bilipschitz.
\end{restatable}

Before the proof, we give some applications.

\subsection{General HHS corollaries}\label{subsec:corollaries-HHS}
We first state some general consequences for HHSs.

\bushyQR*

\begin{proof}
Theorem \ref{thm:main} provides a standard $k$--flat $\bf F$ in $\hatcone$, so by Remark \ref{rem:hedge-flat-dimension}, $\mathfrak S-\mathfrak S_{ql}$ contains a set of pairwise orthogonal elements $U_1,\ldots,U_k$ with each $\mathcal CU_i$ bushy, as required.
\end{proof}

The following notation will be convenient below:

\begin{defn}\label{defn:bushy-QR}
Let $(X,\mathfrak S)$ be an HHS.  Given $\mathfrak T\subseteq\mathfrak S$, the \emph{orthogonality number} $\orthnum(\mathfrak T)$ of $\mathfrak T$ is the maximal cardinality of subsets of $\mathfrak T$ whose elements are pairwise orthogonal. Let $\bushy$ be the set of $U\in\mathfrak S$ such that $\mathcal CU$ is bushy.
\end{defn}

For example, the bound given by Corollary \ref{cor:bushy-QR} is $k\leq \orthnum(\bushy)$.

\begin{remark}\label{rem:standard-bushy-QR}
If $(X,\mathfrak S)$ is standard, then $\orthnum(\bushy)$ is at least the maximal cardinality of sets of pairwise orthogonal elements $U\in\mathfrak S$ such that $\mathcal CU$ has diameter more than $E$ and $U$ is not $\nest$--minimal.
\end{remark}

For HHGs, we get:

\begin{cor}\label{cor:HHG-version}
Let $(G,\mathfrak S)$ be a HHG such that $\stab_G(U)$ acts coboundedly on $\mathcal CU$ for all $U\in\mathfrak S$.  Let $k=\orthnum(\bushy)$. Then $k$ is the maximal $n$ such that $F_2^n$ coarsely embeds in $G$, and $F_2^k$ quasi-isometrically embeds in $G$.
\end{cor}

\begin{proof}
By Lemma \ref{lem:hhg_std}, we can change the HHG structure on $G$ to ensure that it is standard, without affecting $k=\orthnum(\bushy)$. Corollary \ref{cor:bushy-QR} implies that $F_2^n$ can coarsely embed in $G$ only if $n\leq k$. We are left to argue that $F_2^k$ quasi-isometrically embeds into $G$. Let $U_1,\dots, U_k$ be bushy pairwise orthogonal domains. Since $\stab_G(U)$ acts coboundedly on $\mathcal CU$, we have that $\stab_G(U)$ contains a free subgroup of rank two with quasi-isometrically embedded orbits in $\mathcal CU_i$. This implies that $F_{U_i}$ contains a quasi-isometrically embedded copy of $F_2$.  Recalling from Remark \ref{rem:product-region-facts} that $\Pi_{i=1}^k F_{U_i}$ quasi-isometrically embeds into $G$, we are done.
\end{proof}

We are interested in coarse embeddings with image contained in certain subgroups of interest in particular HHGs, such as Torelli subgroups of mapping class groups. Sometimes the maximal number of factors of a product of free groups that can be coarsely embedded in the subgroup is lower than for the ambient group. This is because, roughly, certain domains cannot contribute, and these are captured in the following definition.

\begin{defn}\label{defn:incompatible}
    Let $(X,\mathfrak S)$ be an HHS and let $A\subseteq X$. A subset $\mathcal T$ of $\mathfrak S$ is \emph{incompatible with $A$} if there do not exist sequences $U_n\in\mathcal T$, $x_n,y_n\in P_{U_n}$,such that
    \begin{itemize}
        \item $\lim_n \dist_{U_n}(x_n,y_n)=+\infty$,
        \item $\lim_n \dist_{X}(x_n,y_n)/\dist_{U_n}(x_n,y_n)<+\infty$,
        \item $\lim_n \dist_{X}(x_n,A)/\dist_{U_n}(x_n,y_n)=0$, and same for $y_n$.
    \end{itemize}
    
\end{defn}

We then have the following refinement of Corollary \ref{cor:bushy-QR}.

\begin{cor}
\label{cor:incompatible}
    Let $(X,\mathfrak S)$ be a standard HHS, let $A\subseteq X$, and let $\mathcal T\subseteq \mathfrak S$ be incompatible with $A$.  If $f:T_3^k\to A$ is a coarse embedding, then $k\leq\orthnum(\bushy-\mathcal T)$.
\end{cor}

\begin{proof}
    Let $f:T_3^k\to A$ be a coarse embedding, and consider the flat $\mathbf F$ as in Theorem \ref{thm:main}; we also fix the rest of the notation of Theorem \ref{thm:main}. Denote by $\{\mathbf U_i\}_{i=1}^k$ the support of $\mathbf F$. It suffices to show that for all $i=1,\dots,k$ and for $\omega$-a.e. $n$ we have $U^n_i\notin \mathcal T$. Suppose that this is not the case, that is, up to reordering the elements of the supports, there exists $i$ such that $U^n_i\in \mathcal T$ for $\omega$-a.e. $n$.

 Since the composition ${\bf \mathfrak g}\circ{\bf \hat f}$ is bilipschitz, there exist sequences $\mathbf a,\mathbf b$ in $T^k_3$ such that
\begin{itemize}
\item $\omega-\lim d(a_n,b_n)/r_n=1$,
    \item $\omega-\lim \dist_{\mathcal CU_i^n}(f(a_n),f(b_n))/r_n=\alpha\in (0,\infty)$,
    \item for all $\epsilon> 0$ there exists $a'_n\in T^3_k$ with $\omega-\lim d(a'_n,a_n)/r_n<\epsilon$ and $\omega-\lim \dist_{\mathcal CU_i^n}(f(a'_n),f(a_n))/r_n>0$, and same for $b_n$.
\end{itemize}

Now, set $x_n=\mathfrak g_{H_n}(f(a_n)), y_n=\mathfrak g_{H_n}(f(b_n))$, so that in particular $x_n,y_n\in P_{U^i_n}$. By construction of gates we have $\omega-\lim \dist_{\mathcal CU_i^n}(x_n,y_n)/r_n=\omega-\lim \dist_{\hat X}(f(a_n),f(b_n))/r_n=\alpha$. Keeping into account that $f$ is coarsely lipschitz, we see that subsequences of $(x_n)$ and $(y_n)$ contradict incompatibility provided that $\omega-\lim \dist_X(x_n, P_{U^i_n})/r_n=0$, and same for $y_n$. But this follows from the last bullet point and \cite[Lemma 13.12]{HHS_I}, saying that far from the product region the projection to $\mathcal CU_i^n$ is sublinearly contracting. 
\end{proof}

\subsection{Mapping class group applications}
The following can be proven using an Euler characteristic argument --- see \cite[Lemma 4.2]{Bowditch:WP} --- and is illustrated by Figure \ref{fig:complexity}.

\begin{lemma}
\label{lem:Euler}
    The maximal cardinality of a collection of disjoint subsurfaces in $\Sigma_{g,p}$, each of complexity at least 1, is $\left\lfloor \frac{3g+p-2}{2}\right\rfloor$.
\end{lemma}

\begin{figure}[h]
\includegraphics[scale=0.6]{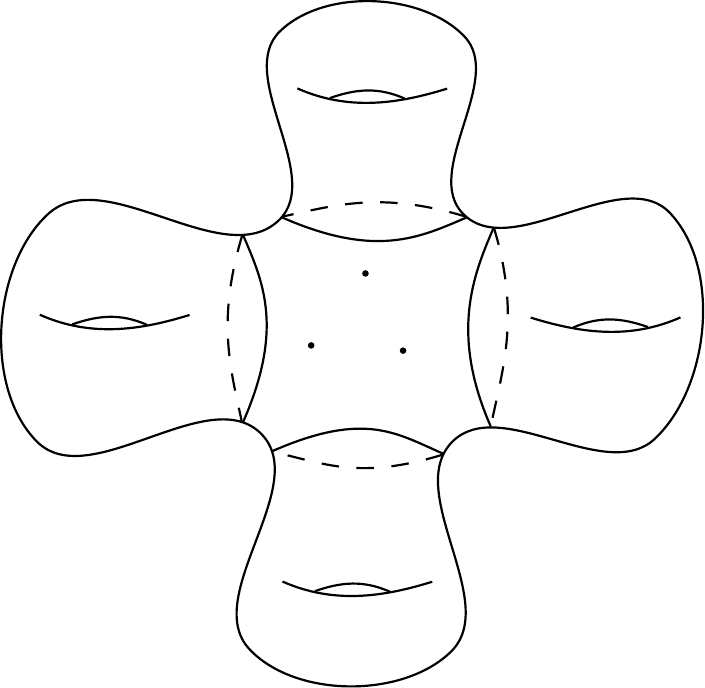}
\caption{Cutting out $g$ one-holed tori from $\Sigma_{g,p}$ yields a sphere with $g+p$ holes/punctures, which contains $\lfloor g+p-2\rfloor$ disjoint 4-holed-spheres, for a total of $\left\lfloor \frac{3g+p-2}{2}\right\rfloor=g+\lfloor \frac{g+p-2}{2}\rfloor$ subsurfaces.}
\label{fig:complexity}
\end{figure}

We now prove our first theorem from the introduction, which we recall for convenience.

\mcg*

\begin{proof}
    Let $k= \left\lfloor \frac{3g+p-2}{2}\right\rfloor$. It suffices to argue that $\MCG(\Sigma_{g,p})$ contains a quasi-isometrically embedded subgroup isomorphic to $F_2^k$, and that $F_2^{k+1}$ does not coarsely embed into $\MCG(\Sigma_{g,p})$.
    
    For the former, by Lemma \ref{lem:Euler} there are $k$ disjoint subsurfaces of complexity at least 1 in $\Sigma_{g,p}$, and we can consider quasi-isometrically embedded free groups supported in each. The product of these quasi-isometrically embeds into $\MCG(\Sigma_{g,p})$ by the distance formula.  The latter follows from Corollary \ref{cor:HHG-version} and Lemma \ref{lem:Euler}.
\end{proof}

We now consider Torelli groups $\mathcal I_g$, to which we will apply Corollary \ref{cor:incompatible}. In order to do so, we need the following lemma.

\begin{lemma}\label{lem:which-matrix-norm}
    Let $\mathcal T$ be the set of all one-holed tori in $\Sigma_g$. Then $\mathcal T$ is incompatible with $\mathcal I_g$.
\end{lemma}

\begin{proof}
    Fix $T\in\mathcal T$, so that any other $T'\in\mathcal T$ can be written as $xT$ for some mapping class $x$. We identify $P_T$ with $\stab(T)$.
    Consider sequences $U_n=x_nT\in\mathcal T$, $x_n,y_n=x_nd_n\in P_{U_n}=x_n P_T$, so $d_n\in\stab(T)$, and suppose by contradiction that all bullet points in Definition \ref{defn:incompatible} are satisfied. Let $\phi:\MCG(\Sigma_g)\to GL_{2g}(\mathbb R)$ be given by the action on homology, where we endow $GL_{2g}(\mathbb R)$ with the $\ell^2$-norm $|| \cdot||$, and we choose a basis for the homology whose first two vectors give a basis for the homology of $T$ coming from two curves $\alpha,\beta$ in $T$.  Let $r_n=\dist_{U_n}(x_n,y_n)$.

\setcounter{claim}{0}
\begin{claim}\label{claim:intersection-number}
    There exists a constant $C>0$ such that $\log ||\phi(d_n)||\geq r_n/C-C$.
\end{claim}

\begin{proofofclaim}{\ref{claim:intersection-number}}
We have $d_n\in \stab(T)$, so $d_n\alpha$ is a curve on $T$. We have $i(\alpha,d_n\alpha)\geq 2^{r_n-C'}$ by \cite[Lem. 2.1]{Hempel}, where $i(\cdot,\cdot)$ is the intersection number and $C'$ is a constant. On the other hand, $i(\alpha,d_n\alpha)$ is the absolute value of a single entry of the matrix $\phi(d_n)$ (the entry corresponding to the component along $\beta$ of $\phi(d_n)([\alpha])$). Therefore $||\phi(d_n)||\geq 2^{r_n-C'}$, and we are done.
\end{proofofclaim}

    The second and third bullet points in Definition \ref{defn:incompatible}, and the fact that $\mathcal A=\mathcal I_g$, imply $\lim \dist_{\MCG(\Sigma_g)}(1,d_n)/r_n=0$, but this contradicts the claim since $\log ||\phi(d_n)||$ is linearly bounded from above by $\dist_{\MCG(\Sigma_g)}(1,d_n)$.
\end{proof}

We are now ready to prove our second main theorem.

\torellijohnsonthm*

\begin{proof}
It suffices to show that $F_2^{g-1}$ coarsely embeds into $\mathcal J_g$, and that $F_2^g$ cannot coarsely embed into $\mathcal I_g$.

For the former, we can find $g-1$ disjoint 4-holed spheres in $\Sigma_g$, each containing an essential curve which is separating in $\Sigma_g$, see Figure \ref{fig:spheres}.

\begin{figure}[h]
\includegraphics[]{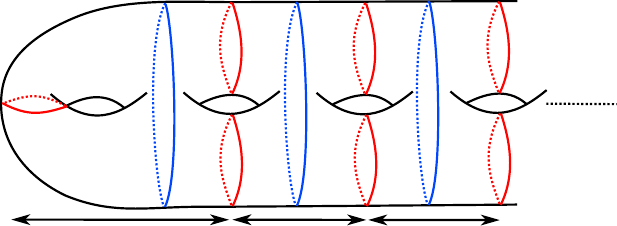}
\caption{The figure illustrates the pattern required to fit $g-1$ 4-holed spheres, each containing a separating curve, on $\Sigma_g$. Three 4-holed spheres are indicated with double arrows.}\label{fig:spheres}
\end{figure}

Dehn-twisting this curve gives another such curve inside each 4-holed sphere. Powers of the Dehn twists around these two curves generate a free group, and the free groups coming from different 4-holed spheres commute, yielding a $F_2^{g-1}$ subgroup, as required.

To show that $F_2^g$ cannot coarsely embed into $\mathcal I_g$ we will use Corollary \ref{cor:incompatible} with $X$ the mapping class group and $A=\mathcal I_g$ (note that the corollary applies to mapping class groups by Lemma \ref{lem:hhg_std}). For $\mathcal T$ we take the set of all the 1-holed tori, which is incompatible with $\mathcal I_g$ by Lemma \ref{lem:which-matrix-norm}. In view of Corollary \ref{cor:incompatible}, we are left to argue that $\Sigma_g$ does not contain $g$ pairwise disjoint subsurfaces of complexity a least $1$, none of which is a one-holed torus. Any such subsurface has Euler characteristic at most $-2$, so the conclusion easily follows, using that the Euler characteristic is additive under gluing subsurfaces along circles.
\end{proof}

We now move on to surface braid groups.

\begin{lemma}
\label{lem:braid_prod_free}
For $g\geq 3$ and $n\geq 5$,  there is  a subgroup of $B_n(\Sigma_g)$ isomorphic to $F_2^{k}$, where $k=2+\lceil (n-1)/2\rceil$. Moreover, $B_n(\Sigma_n)$ contains a subgroup isomorphic to $F_2^n$. 
\end{lemma}

\begin{proof}
For the first part, Figure \ref{fig:braids} illustrates how to subdivide $\Sigma_g$ into $3+\lceil (n-3)/2\rceil =2+\lceil (n-1)/2\rceil$ disjoint subsurfaces each supporting a free subgroup of $B_n(\Sigma_g)$. As usual mapping classes supported on disjoint subsurfaces commute.
\begin{figure}[h]
    \includegraphics[scale=0.8]{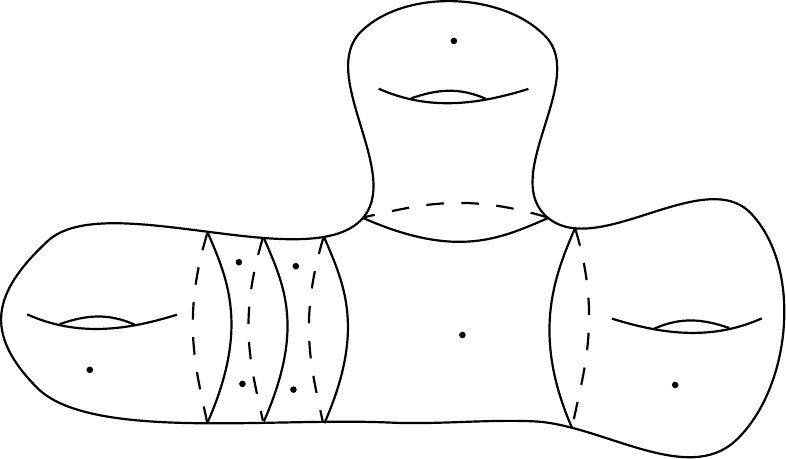}
    \caption{Three punctures get placed in a subsurface with non-zero genus, and the others get distributed in pairs into annuli, and if one remains it gets placed in the pair of pants in the middle of the figure.}\label{fig:braids}
\end{figure}

    For the second part, $\Sigma_{n,n}$ contains $n$ once-punctured and one-holed tori, and these are each the support of a free subgroup of $B_n(\Sigma_g)$.
\end{proof}

\braids*

\begin{proof}
By Lemma \ref{lem:braid_prod_free}, $B_n(\Sigma_n)$ contains a subgroup isomorphic to $F_2^n$, so we are left to argue that $B_n(\Sigma_2)$ does not contain a coarsely embedded copy of $F_2^n$. This is because, by Theorem \ref{thm:mcg_bushy_rank}, $\MCG(\Sigma_{n,n})$ does not, and $B_n(\Sigma_2)$ is a subgroup.
 \end{proof}

Recall that $\mathcal J_g$ is known to be finitely generated for $g\geq 4$ \cite{CEP:J_fin_gen} (see also \cite{EH:J_fin_gen}), but this is unknown for $g=3$.

\begin{proof}[Proof of Corollary \ref{cori:torelli_vs_braids}]
   The cohomological dimension of $\mathcal I_g$ is $3g-5$, so if it was quasi-isometric to a surface braid group $B_n(\Sigma_g)$ we would have $n=3g-5$ in view of \cite{dim_surface_braid}. But by Theorem \ref{thm:Torelli_Johnson} and Lemma \ref{lem:braid_prod_free} we would need to have
   $$\lfloor (3g-5-1) /2\rfloor -1 \leq g-1,$$
   which does not hold for $g\geq 6$.

   The cohomological dimension of $\mathcal J_g$ is $2g-3$ \cite{dimension_torelli}, while that of $B_n(\Sigma_h)$ is $n$ \cite{dim_surface_braid}. Hence, if those are quasi-isometric, we must have $n=2g-3$, which is at least 5. However, by Theorem \ref{thm:Torelli_Johnson} and Lemma \ref{lem:braid_prod_free}, we should also have 
$$ 2+\lceil (2g-4)/2 \rceil \leq g-1,$$
but the left-hand side is equal to $g$, a contradiction.
 \end{proof}

\subsection{Bestvina-Brady groups}

The following is a more precise version of Theorem \ref{thmintro:BB}.

\begin{thm}\label{thm:bestvina-brady}
For all $n\geq 2$ and arbitrarily large $m\gg n$, there exists a finite simplicial graph $\Gamma$ such that the Bestvina-Brady kernel $B(\Gamma)$ has the following properties:
\begin{itemize}
    \item $B(\Gamma)$ is of type $F_2$ but not $FP_3$,
    \item $B(\Gamma)$ has Dehn function $\asymp n^4$,
    \item $asdim(B(\Gamma))=m$,
    \item $F_2^k$ coarsely embeds in $B(\Gamma)$ if and only if $k\leq n$.
\end{itemize}
\end{thm}

\begin{proof}
By \cite{BestvinaBrady} the first property holds provided that the flag complex $L(\Gamma)$ associated to $\Gamma$ is simply connected but not homologically 2-connected. We will arrange $L(\Gamma)$ to be homotopically equivalent to a 2-sphere.

Dehn functions of Bestvina-Brady kernels with $L(\Gamma)$ simply connected are fully classified in \cite[Thm. 1]{BB_Dehn}. In particular, the Dehn function is quartic provided that $\Gamma$ contains a maximal join $\Lambda$ such that $L(\Lambda)$ is not simply connected. 

So, let $\Lambda$ be a square, and let $\Gamma'$ be the $1$--skeleton of a triangulation of $S^2$ such that $\Lambda$ embeds in $\Gamma'$ in such a way that  $\Lambda$ is not contained in the star of a vertex.  

Now let $\Delta_{m'}$ be an $(m'+1)$--clique, and let $\Lambda_n$ be the cone over the $n$--fold join of the discrete graph with two vertices.  Let $\Gamma$ be the wedge of $\Gamma',\Delta_{m'},\Lambda_n$.  

Then $L(\Gamma)$ is homotopy equivalent to $L(\Gamma')$, and hence to $S^2$, since $L(\Lambda_n)$ and $L(\Delta_{m'})$ are contractible.  Hence $B(\Gamma)$ is of type $F_2$ but not $FP_3$.  Since passing from $\Gamma$ to $\Gamma'$ did not introduce any vertex whose star contains $\Lambda$, the Dehn function of $B(\Gamma)$ is still quartic. Since the asymptotic dimension of $A_{\Gamma}$, for $m'$ large enough, is $m'$, the asymptotic dimension of $B(\Gamma)$ is either $m'$ or $m'-1$ by the Hurewicz theorem for asymptotic dimension \cite{Hurewicz_asdim}. 

Finally, the maximal $k$ such that $F_2^k$ coarsely embeds in $B(\Gamma)$ is $n$.  Indeed, the Bestvina-Brady map $\phi:A(\Gamma)\to\integers$ whose kernel is $B(\Gamma)$ restricts on $A(\Lambda_n)\subset A(\Gamma)$ to the corresponding map defined intrinsically on $A(\Lambda_n)$.  Now, $A(\Lambda_n)\cong \langle a\rangle \times \prod_{i=1}^n\langle x_i,y_i\rangle$, where $x_i,y_i,a$ are the vertices of $\Lambda_n$.  Note that $\prod_{i=1}^n\ker(\phi|_{\langle x_i,y_i\rangle})\subseteq \ker(\phi)$, so $\ker(\phi)=B(\Gamma)$ contains $F_2^n$ since each $\phi|_{\langle x_i,y_i\rangle}$ is free of rank at least $2$.

On the other hand, orthogonal bushy domains in the HHS structure of $A_\Gamma$ described in \cite{HHS_I} correspond to joins of subgraphs of diameter at least $2$. Therefore, by construction, there are at most $n$ pairwise orthogonal bushy domains for $A_\Gamma$, so Corollary \ref{cor:bushy-QR} applies.
\end{proof}

\section{Counting}\label{subsec:counting}
From now until the end of the paper, fix a standard HHS $(X,\mathfrak S)$ and let $q:(X,\mathfrak S)\to (\widehat X,\mathfrak S-\mathfrak S_{ql})$ be the factored space from Definition \ref{defn:standard-HHS}.  The metrics on $X$ and $\widehat X$ are respectively $\dist$ and $\hat\dist$.

\begin{assume}\label{assume:X}
By passing to a fixed net in $X$, we can and shall assume that $X$ is a discrete quasigeodesic space, and $X$ has bounded geometry, by Definition \ref{defn:standard-HHS}, so there is a function $\growth:[0,\infty)\to\integers_{\geq0}$ such that for all $R\geq 0$, any $R$--ball in $X$ has cardinality at most $\growth(R)$, and there exists $a_0>1$ such that $\growth(R)\leq a_0^R$ for $R\geq 0$.  By rescaling the metric, we can and shall assume that $\dist(x,y)\geq 1$ whenever $x,y\in X$ are distinct.  Hence the same holds for the metric $\hat\dist$ on $\widehat X$, by \cite[Prop. 2.2]{HHS_III}.
\end{assume}

Let $T_3$ be the $3$--regular simplicial tree, with the usual combinatorial path-metric, denoted $d$.  In this section, we abuse notation and conflate $T_3$ with its vertex set and view $(T_3,d)$ as a discrete space. Fix $k\geq 1$ and let $d$ also denote, say, the $\ell_1$ metric on $T_3^k$ (arguments in this section would be unchanged if $d$ was the $\ell_p$ metric for some other $p\in [1,\infty]$).

Fix a coarse embedding $f:(T_3^k,d)\to (X,\dist)$ and let $\hat f=q\circ f$.

\begin{defn}\label{defn:f-constants}
Let $N_0\in\naturals$ be such that $|f^{-1}(x)|\leq N_0$ for all $x\in X$.  Let $L\geq 1$ be a lipschitz constant for $f$.
\end{defn}

The goal of this subsection is to prove the following proposition:

\begin{prop}\label{prop:counting}
For all $a>1$, there exist $\epsilon,r_0>0$ such that the following holds. Let $o\in T_3^k$ and $r>r_0$. Then
$$\{v\in T_3^k: d(o,v)\leq r, \hat\dist(\hat f(o),\hat f(v))<\epsilon r\}$$
has cardinality at most $a^r$.
\end{prop}

We prove the proposition at the end of the section, after obtaining the needed lemmas.  In the next lemma, note the difference with Proposition \ref{prop:counting}: we count $v\in T_3^k$ that get sent $\epsilon_2r$--close to $f(o)$ as measured in $X$, not $\widehat X$.  

\begin{lem}\label{lem:easy-count}
For all $b>1$ there exist $\epsilon_2,r_0>0$ such that the following holds. Let $o\in T_3^k$ and $r>r_0$. Then
$$\{v\in T_3^k: d(o,v)\leq r, \dist(f(o),f(v))<\epsilon_2 r\}$$
has cardinality at most $b^r$.
\end{lem}

\begin{proof}
For any $\epsilon>0$ and $r\geq 0$, the set of $v\in T_3^k$ such that $\dist(f(o),f(v))\leq \epsilon r$ has cardinality at most $N_0\growth(\epsilon r)\leq N_0 a_0^{\epsilon r}$.  Choose $\epsilon\in(0,1)$ so that $a_0^\epsilon<b$ and choose $r_0$ so that $N_0<(b/a_0^\epsilon)^r$ for $r>r_0$.
\end{proof}

Below, we consider a pair of points $x,y\in X$ whose images in $\widehat X$ are much closer than the original distance $d(x,y)$ in $X$, as is the situation, say, with $x=f(o)$ and $y=f(v)$ for $o,v$ from Proposition \ref{prop:counting}.  Fixing $x$, we are interested in the number of possibilities for $y$.  Lemma \ref{lem:sequence} below will give us a template for counting them, by producing a sequence $x=x_0,\ldots,x_n=y$ such that, for odd $i$, the points $x_i,x_{i+1}$ differ significantly on at most one projection, to some $U\in\mathfrak S_{ql}$, and the total of these $\dist(x_i,x_{i+1})$ is linear in $\dist(x,y)$; meanwhile, the total of the $\dist(x_i,x_{i+1})$ for even $i$ is a tiny proportion of $\dist(x,y)$.  For $i$ odd, Lemma \ref{lem:linear} will constrain the number of choices for $x_{i+1}$ given $x_i$.  So, roughly speaking, this will give us a bound on how many sequences $x_0,\ldots x_n$ (and hence how many choices for $y$) are possible, given the sequence of distances $d(x_i,x_{i+1})$ for $i$ odd, which sum to at most a fixed multiple of $d(x,y)$.  The remaining thing to consider is how many such sequences of distances there are, which motivates the following definition and lemma.

\begin{defn}\label{defn:partitions}
Let $s\in\integers_+, n\in \reals_+$ and let $\partition(s,n)$ be the set of tuples $(a_1,\ldots,a_m)$, $m\leq n$, of positive integers such that $\sum_{i=1}^m a_i= s$.  Let $P(s,n)=|\partition(s,n)|$.
\end{defn}

\begin{lem}\label{lem:partition-counting}
    For all $c>1$ and all sufficiently small $\epsilon>0$,  $P(s,\epsilon s)\leq c^s$ for all  $s>0$.
\end{lem}

\begin{proof} 
Let $s\in\integers_+$ and $n\in\reals_+$.  For each $m\geq 1$, the number of tuples $(a_1,\ldots,a_m)$ with all $a_i\geq 1$ and $\sum_{i=1}^ma_i=s$ is at most ${s-1\choose m-1}$.  Indeed, consider a string $w\in\{0,1\}^*$ consisting of $a_1$ $1$s, then a single $0$, then $a_2$ $1$s, then a single $0$, etc., with a total of $s$ $1$s, a total of $m-1$ $0$s, and no substring $00$.  So $w$ is obtained from the string of $s$ $1$s by replacing $m-1$ of the $s-1$ substrings $11$ with $101$; there are ${s-1\choose m-1}$ ways to do this.  Hence
$$P(s,n)\leq \sum_{m=0}^{\lfloor n\rfloor}{s\choose m}.$$
Let $H:(0,1)\to\reals$ be the binary entropy function, $H(\epsilon)=-\epsilon\log_2\epsilon-(1-\epsilon)\log_2(1-\epsilon)$.  Then $2^{H(\epsilon)s}\leq c^s$ for all $s>0$ provided $\epsilon\in(0,\frac12)$ is sufficiently small in terms of $c$.  For such $\epsilon$, a standard estimate for sums of binomial coefficients (\cite[Lem. 16.19]{FlumGrohe}) gives 
$$P(s,\epsilon s)\leq \sum_{m=0}^{\lfloor\epsilon s\rfloor}{s\choose m}\leq 2^{H(\epsilon)s}<c^s,$$
as required.
\end{proof}

Next, we produce the points $x_0,\ldots,x_n$ mentioned above.

\begin{lem}\label{lem:sequence}
There exist constants $M,K$, with $1\leq K\leq M$, and $\epsilon_0>0$, such that for all $\epsilon\in(0,\epsilon_0)$ the following holds.
    Let $x,y\in X$ be such that $\hat\dist(x,y)<\epsilon \dist(x,y)$. Then there exist $n\leq M\epsilon \dist(x,y)$ and a sequence $x=x_0,\dots,x_n=y$ such that
    \begin{enumerate}[(A)]
        \item $$\sum_{i \text{ even}} \dist(x_i,x_{i+1})\leq M\epsilon \dist(x,y),$$\label{item:even-sum-small}
        \item $$\sum_{i \text{ odd}} \dist(x_i,x_{i+1})\leq M \dist(x,y),$$\label{item:odd-sum-linear}
        \item for $i$ odd, $\relevant_K(x_i,x_{i+1})$ consists of a single $U_i\in\mathfrak S_{ql}$, so $\relevant_M(x_i,x_{i+1})\subseteq \{U_i\}$.\label{item:odd-one-relevant}
    \end{enumerate}
\end{lem}
\begin{proof}
By Standing Assumption \ref{assume:X}, if $\epsilon>0$ and $x,y\in X$ satisfy $\hat\dist(x,y)< \epsilon\dist(x,y)$, then $x\neq y$ and $\dist(x,y)>1/\epsilon$; we use this freely below.

\textbf{Constants.}  Below, we will use a constant $C\ge 1$ that is chosen sufficiently large in terms of the HHS parameters only.  As we go, we will point out the conditions that $C$ must satisfy, and note that these only depend on the input HHS structures, and could have been imposed at the beginning of this proof. To begin, we require that $C\geq \max\{s_{Id},10E\}$. Let $K=100(C+E+\theta+1)$.  Let $A=\DF(K)$.  Choose $\epsilon_0\in(0,1)$ so that $2A^2\epsilon_0<1$ and $\epsilon_0<\min_{s_{Id}\leq K\leq K_1}\DF(K)^{-1}$.
\setcounter{claim}{0}

\begin{claim}\label{claim:nonempty-relevant}
Let $\epsilon\in(0,\epsilon_0)$. Let $x,y\in X$ satisfy $\hat\dist(x,y)<\epsilon d(x,y)$.  Then 
$$\mathfrak S_{ql}\cap \relevant_T(x,y)\neq\emptyset$$
for all $T\leq K$.
\end{claim}

\begin{proofofclaim}{\ref{claim:nonempty-relevant}}
Let $P:=\sum_{U\in\mathfrak S_{ql}}\ignore{\dist_U(x,y)}{K}$ and $Q:=\sum_{U\in\mathfrak S-\mathfrak S_{ql}}\ignore{\dist_U(x,y)}{K}$. 
If $P=0$ then the distance formula and the assumption $\hat\dist(x,y)<\epsilon\dist(x,y)$ imply $A^{-1}\dist(x,y)-A\leq Q<\epsilon\dist(x,y)$, so since our choice of $\epsilon_0$ implies $\dist(x,y)>1/\epsilon>2A^2$, we have $\dist(x,y)/2A\leq Q<\epsilon\dist(x,y)$, so $2A\epsilon>1$, contradicting our choice of $\epsilon$.
 
Hence $P>0$, so $\mathfrak S_{ql}\cap\relevant_{K}(x,y)\neq\emptyset$, proving the claim. 
\end{proofofclaim}

\textbf{Ordering relevant domains.} 
Fix $\epsilon$ and $x,y\in X$ as in Claim \ref{claim:nonempty-relevant}.  Let $\mathfrak L_T(x,y)=\mathfrak S_{ql}\cap\relevant_K(x,y)$ for $T\leq K$, so Claim \ref{claim:nonempty-relevant} implies $\mathfrak L_T(x,y)\neq\emptyset$.  

By Definition \ref{defn:standard-HHS}, any two distinct elements of $\mathcal L_C(x,y)$ are orthogonal or transverse.  Hence, by \cite[Prop. 2.8]{HHS_II} and the fact that $C>E$, there is a partial order $\prec$ on $\mathcal L_C(x,y)$ such that $U\precneq V$ if and only if $U\transverse V$ and $\dist_U(\rho^V_U,y)\leq E$ (equivalently, by our choice of $M_0$ and \cite[Defn. 1.1.(4)]{HHS_II}, $\dist_V(\rho^U_V,x)\leq E$), and $U,V\in\mathcal L_C(x,y)$ are $\prec$--incomparable if and only if $U\orth V$.  

Write $\mathfrak L_K(x,y)=\{U_{2t-1}\}_{t=1}^m$, where the $U_i$ are labelled (by odd $i$) according to the following inductive construction.  First, let $U_1,\ldots,U_{2m_1-1}$ be all of the $\prec$--minimal elements of $\mathfrak L_K(x,y)$.  Then let $U_{2m_1+1},\ldots,U_{2m_2-1}$ be the $\prec$--minimal elements of $\mathfrak L_K(x,y)-\{U_1,\ldots,U_{2m_1-1}\}$.  Continuing in this way, we have ordered the elements of $\mathfrak L_K(x,y)$ so that:
\begin{itemize}
    \item if $1\leq i<j\leq 2m-1$, then $U_i\precneq U_j$ or $U_i\orth U_j$, 
    \item if $i>1$ is odd, then there exists odd $j<i$ such that $U_j\precneq U_i$ (in particular, $U_i\transverse U_j$) and $j\leq i+2\chi$, were $\chi$ is bounded by the maximal cardinality of a pairwise orthogonal subset of $\mathfrak S$.
\end{itemize}
For each odd $i\leq 2m-1$, let $\gate_i:X\to P_{U_i}$ be the gate map.

\textbf{Constructing points $x_i$.}  Construct a set $\mathcal X_K$ of points $x_i$ as follows.  First, let $x_0=x$ and let $x_1=\gate_1(x)$.  If $i\geq 1$ is odd and $x_i$ has been constructed already, then let $x_{i+1}'=\gate_{i}(y)$.  

By the realisation theorem, \cite[Thm. 3.1]{HHS_II}, since $C$ is sufficiently large there exists $x_{i+1}\in P_{U_{i}}$ such that $\dist_V(x_{i+1},x'_{i+1})\leq C$ for $V\not\perp U_{i}$ and $\dist_V(x_{i+1},x_{j(i)})\leq C$ for $V\orth U_{i}$, where $j(i)-1<i$ is the maximal odd integer with the property that $U_{j(i)}\transverse V$ or $U_{j(i)}\propnest V$.  If no such $j(i)$ exists, use $x_{j(i)}=x_i$.

\begin{figure}[h]
  \begin{overpic}[width=0.65\linewidth]{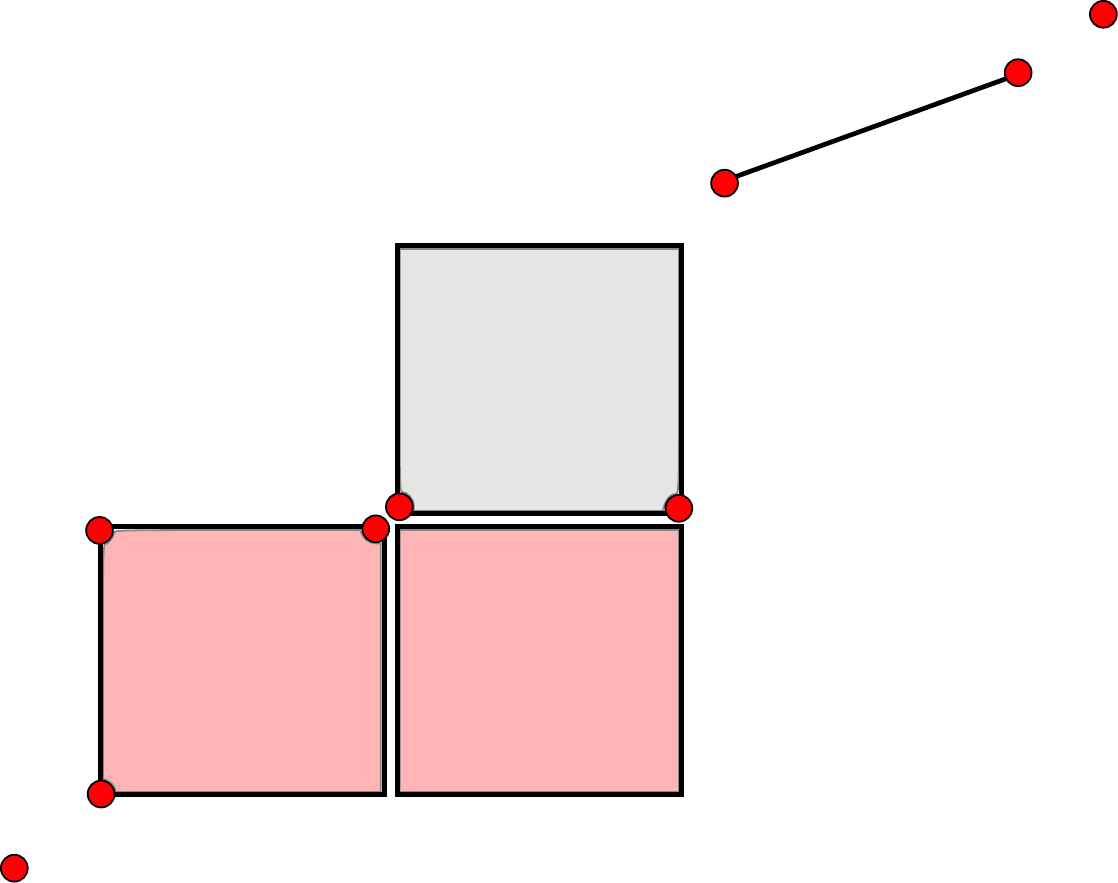}
  \put(-12,1){$x=x_0$}
  \put(3,8){$x_1$}
  \put(2,18){$U_1$}
  \put(3,34){$x_2\asymp x_3$}
  \put(22,27){$U_3$}
  \put(45,27){$U_5$}
  \put(29,34){$x_4$}
  \put(37,36){$x_5$}
  \put(63,34){$x_6$}
  \put(63,65){$x_7$}
  \put(76,70){$U_7$}
  \put(91,68){$x_8$}
  \put(63,45){$V$}
  \put(101,77){$x_9=y$}
  \end{overpic}
    \caption{Constructing $\mathcal X_K$.  The picture takes place in (a neighbourhood of) $\hull(x,y)$.  The red squares lie in standard product regions $F_{U_i}\times F_{U_j}$ for various orthogonal pairs $U_i,U_j$, and the grey square is in $F_{U_5}\times F_V$ for some $V\orth U_5$ with $V\in\mathfrak S-\mathfrak S_{ql}$. The $\prec$--partial order goes from left to right and bottom to top, and $U_1\orth U_3,U_1\orth U_5, U_3\prec U_5$, and $U_1,U_3,U_5\prec U_7$.  We construct $x_i$ as follows: start with $x=x_0$ and gate to $P_{U_1}$ to get $x_1$.  The gate of $y$ to $P_{U_1}$ is coarsely the same as the point marked $x_6$, and we move in the orthogonal complement of $U_1$ to get $x_2$.  Since $U_1\orth U_3$, $x_2$ is already (coarsely) in $P_{U_3}$, so doesn't move (much) when we gate, to get $x_3$.  Continuing in this way produces the red points $x_i,\ 0\leq i\leq 9$.}
    \label{fig:X_K}
\end{figure}
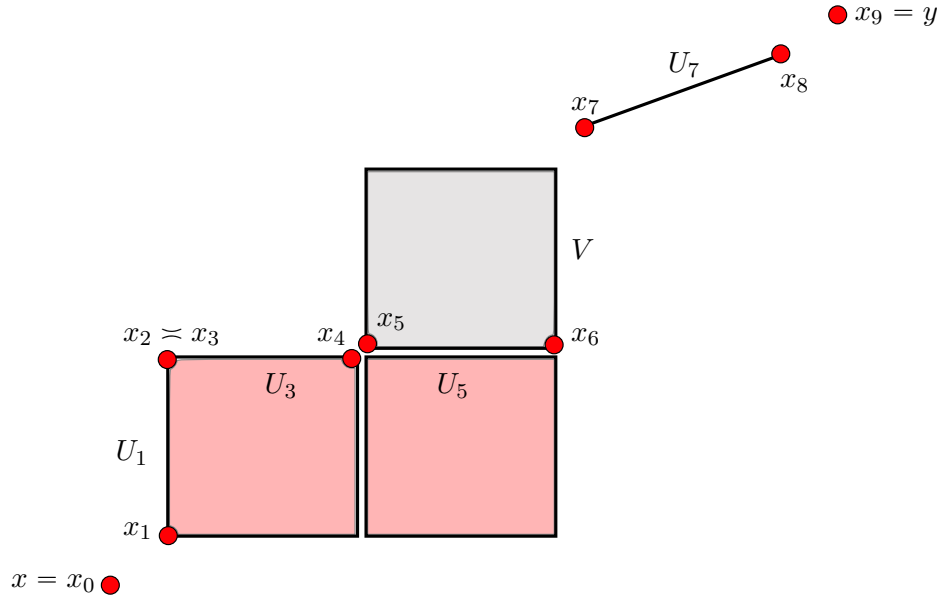

Now suppose that $i\geq 2$ is even and $x_i$ has been constructed already and let $x_{i+1}=\gate_{i+1}(x_i)$. This produces $x_0,x_1,\ldots, x_{2m}$, and we take $x_{2m+1}=y$, and let $\mathcal X_K=\{x_0,\ldots,x_{2m+1}\}$.  

Note that for each odd $i$, we have $x_i,x_{i+1}\in P_{U_i}$, so for $i'\in\{i,i+1\}$ and $V\in\mathfrak S$ with $U_i\propnest V$ or $U_i\transverse V$, we have $\dist_V(\rho^{U_i}_V,x_{i'})\leq E$.

\begin{claim}\label{claim:odd-x_i-at-start}
Let $i\in\{1,\ldots,2m-1\}$ be odd.  Then $\dist_{U_i}(x_i,x_{i'})\leq 10(C+E)$ for $i'\leq i$ and $\dist_{U_i}(x_{i+1},x_{i'})\leq 10(C+E)$ for $i'\geq i+1$.
\end{claim}

\begin{proofofclaim}{\ref{claim:odd-x_i-at-start}}
If $i=1$, then $\dist_{U_1}(x_0,x_1)\leq C$ and $\dist_{U_1}(x_2,y)\leq C$ by construction, provided $C$ is sufficiently large in terms of the HHS parameters.

Assume that $i>1$ is odd.  For any odd $j<i$ such that $U_i\precneq U_i$, and $i'\in \{j,j+1\}$, we have $\dist_{U_i}(\rho^{U_j}_{U_i},x_{i'})\leq E$, while $\dist_{U_i}(\rho^{U_j}_{U_i},x)\leq E$ since $U_j\precneq U_i$.  Hence $\dist_{U_i}(x_{i'},x_i)\leq 3E$.   

If $j<i$ and $U_i\orth U_j$, then by construction we have $\dist_{U_i}(x_j,x_{j(j)})\leq C$.  Also by construction, $\dist_{U_i}(x_i,\rho^{U_{j(j)}}_{U_i})\leq C$, since $U_{j(j)}\precneq U_i$.  Hence $\dist_{U_i}(x_i,x_j)\leq 5(C+E)$, as needed.

This shows $\dist_{U_i}(x_i,x_{i'})$ is uniformly bounded for $i'\leq i$.  A very similar argument proves the part of the claim about $x_{i+1}$.
\end{proofofclaim}

\textbf{Properties of $\mathcal X_K$.}  We now establish some more facts about $\mathcal X_K$ and $\{U_1,\ldots,U_{2m-1}\}$.

\begin{claim}\label{claim:odd-relevant}
Suppose that $2C+1\leq T\leq K-20(C+E)$.  Then $\relevant_T(x_i,x_{i+1})=\{U_i\}$ for all odd $i\in\{1,\ldots,2m-1\}$. Hence
$$\DF(2C+1)^{-1}\cdot \dist_{U_i}(x_i,x_{i+1})\leq \dist(x_i,x_{i+1})\leq \DF(2C+1)\cdot \dist_{U_i}(x_i,x_{i+1}).$$
\end{claim}

\begin{proofofclaim}{\ref{claim:odd-relevant}}
By Claim \ref{claim:odd-x_i-at-start}, $\dist_{U_i}(x_i,x_{i+1})\geq \dist_{U_i}(x,y)-20(C+E)$, which implies $U_i\in\relevant_{T}(x_i,x_{i+1})$.  Now suppose that $V\neq U_i$.  First, Since $U_i$ is $\nest$--minimal, we cannot have $V\propnest U_i$.  Second, by construction, $x_i,x_{i+1}\in P_{U_i}$, so $\dist_V(x_i,x_{i+1})\leq 2C$ if $U_i\propnest V$ or $U_i\transverse V$.  The remaining possibility is $V\orth U_i$, but $\dist_V(x_i,x_{i+1})\leq C$ for such $V$, by construction.  Thus $\dist_V(x_i,x_{i+1})\leq 2C$ unless $V=U_i$.  We have shown $\relevant_{2C+1}(x_i,x_{i+1})=\{U_i\}$, so the last assertion follows from the distance formula.
\end{proofofclaim}

\begin{claim}\label{claim:large-link-application}
There exists $N$, depending only on the HHS parameters, such that $|\mathfrak L_{100E}(x,y)|\leq \frac{(N-1)}{2}\epsilon\dist(x,y)$, and hence $|\mathcal X_K|\leq N\epsilon\dist(x,y)$.
\end{claim}

\begin{proofofclaim}{\ref{claim:large-link-application}}
Let $\mathcal U=\mathfrak L_{100E}(x,y)$.  Given $V\in \mathfrak S$, say that $U\in\mathcal U$ is \emph{associated to} $V$ if $U\propnest V$, and $\dist_V(x,y)>100E$, and $U\propnest W\propnest V$ implies $\dist_W(x,y)\leq 100E$.  Let $\mathcal W$ be the set of all $V$ such that some $U\in\mathcal U$ is associated to $V$, and let $\mathcal U_V$ be the set of $U$ associated to $V$. Since $\mathfrak S_{ql}$ consists of $\nest$--minimal elements, $\mathcal W\subseteq \mathfrak S-\mathfrak S_{ql}$.  

By the ``passing up large projections'' lemma, which is \cite[Lem. 2.5]{HHS_II}, there exists $N'$, depending only on the HHS parameters, such that for any subset $\mathcal V$ of $\mathcal U$ of size $N'$, there exists $V\in\mathcal W$ such that $\mathcal U_V\cap\mathcal V\neq\emptyset$ (that is, some element of $\mathcal V$ is associated to some element of $\mathcal W$).  Hence there exists $\mathcal U'\leq \mathcal U$ such that $|\mathcal U|\leq N'|\mathcal U'|$ and $\mathcal U'\subseteq \bigcup_{V\in\mathcal W}\mathcal U_V$.

Now fix $V\in \mathcal W$.  By the large link axiom \cite[Defn. 1.1]{HHS_II}, there exists $\kappa$, just depending on the HHS parameters, and $\ell\leq \kappa \dist_V(x,y)$ and $W_1,\ldots W_\ell\propnest V$ such that any $U\in\mathcal U$ with $U\propnest V$ must be nested in some $W_i$.  If $U\in\mathcal U_V$, then either $U=W_i$ for some $i$, or $U\propnest W_i$ for some $i$.  Now, $\dist_{W_i}(x,y)\leq 100E$, or else we contradict that $U\in\mathcal U_V$.  Hence another application of the ``passing up'' lemma bounds the number of elements of $\mathcal U_V$ nested in each $W_i$ in terms of the HHS parameters, so up to uniformly increasing $\kappa$, we have $|\mathcal U_V|\leq \kappa\cdot\dist_V(x,y)$.  Hence
\begin{eqnarray*}
    |\mathcal U|\leq N'\sum_{V\in\mathcal W}|\mathcal U_V|
    &\leq&\kappa N'\sum_{V\in\mathfrak S-\mathfrak S_{ql}}\ignore{\dist_V(x,y)}{100E}\\
    &\leq&2\kappa N'\DF(100E)^2\hat\dist(x,y)<2\kappa N'\DF(100E)^2\cdot\epsilon\dist(x,y),
\end{eqnarray*}
using the distance formula and the fact that $\hat\dist(x,y)\geq 1$. This bound implies the claim.
\end{proofofclaim}

\begin{claim}\label{claim:even-i-sum}
Suppose that $T\geq 80(E+C)+4\theta+1$.  Then for all $V\in\mathfrak S$, the following holds:
\begin{itemize}
    \item For all $i$, we have $\relevant_{2T}(x_i,x_{i+1})\subseteq \relevant_T(x,y)$.  In particular, $\relevant_{2K+1}(x_i,x_{i+1})\cap\mathfrak S_{ql}=\emptyset$ for $i$ even, provided $K\geq T$.

    \item $\sum_{i\ \text{even}}\ignore{\dist_V(x_i,x_{i+1})}{2T}\leq 2\dist_V(x,y).$
\end{itemize}
\end{claim}

\begin{proofofclaim}{\ref{claim:even-i-sum}}
For each $V\in\mathfrak S$, let $\gamma_V$ be a geodesic from $\pi_V(x)$ to $\pi_V(y)$.  Let $i$ be even.  We first bound $\dist_V(x_i,\gamma_V),\dist_V(x_{i+1},\gamma_V)$ by $20(E+C)$, via the following steps.

\begin{itemize}
      
    \item If $V=U_{i-1}$, then by Claim \ref{claim:odd-x_i-at-start}, $\dist_V(x_i,y)\leq 10(C+E)$, and, by construction, $\dist_V(x_{i+1},y)\leq 10(C+E)$.

    \item If $U_{i-1}\transverse V$, then $\dist_V(\rho^{{U_{i-1}}}_V,x_i)\leq C$ since $x_i\in P_{U_{i-1}}$.  Hence if $\dist_V(x_i,\{x,y\})>2E+C$, then by consistency, $\dist_{U_{i-1}}(x,\rho^V_{U_{i-1}}),\dist_{U_{i-1}}(y,\rho^V_{U_{i-1}})\leq E$, contradicting that $\dist_{U_{i-1}}(x,y)\geq K$.  Hence $\dist_V(x_i,\{x,y\})\leq 2E+C$. If $U_{i+1}\transverse V$ then, similarly, $\dist_V(x_{i+1},\{x,y\})\leq 2E+C$.

    \item If $U_{i-1}\propnest V$, then $\dist_V(\rho^{U_{i-1}}_V,x_i)\leq C$, and hence consistency and bounded geodesic image, and Claim \ref{claim:odd-relevant}, imply $\dist_V(x_i,\gamma_V)\leq C+2E$.  Similarly, if $U_{i+1}\propnest V$, then $\dist_V(x_{i+1},\gamma)\leq C+2E$. 

    \item If $U_{i-1}\not\perp V$ and $U_{i+1}\orth V$, then $\dist_V(x_{i+1},x_i)\leq C$ since $x_{i+1}=\gate_{i+1}(x_i)$.  Hence $\dist_V(\gamma_V,x_{i+1})\leq 20(E+C)$.

    \item If $U_{i+1}\not\perp V$ and $U_{i-1}\orth V$, then from above we have $\dist_V(\gamma_V,x_{i+1})\leq 10(E+C)$.  Also, by construction, $\dist_V(x_i,x_j)\leq C$ for all $j_0\leq j<i$, where either $j_0=0$ or $j_0$ is the minimal odd integer with $U_{j_0}\orth V$.  Hence either $\dist_V(x,x_i)\leq C$, or $U_{j_0-2}$ exists and is not orthogonal to $V$, so $\dist_V(x_{j_0-1},\gamma_V)\leq C+2E$, so $\dist_V(x_i,\gamma_V)\leq 3C+2E$.
\end{itemize}
Hence, as long as $T\geq 80(E+C)+4\theta+1$, Lemma \ref{lem:initial-M} implies $\relevant_{2T}(x_i,x_{i+1})\subseteq \relevant_T(x,y)$.  In particular,  our choice of $K$ is large enough that we can conclude that $\relevant_{2K+1}(x_i,x_{i+1})\subseteq \relevant_{K}(x,y)$, so $\relevant_{2K+1}(x_i,x_{i+1})\cap\mathfrak S_{ql}\subseteq \{U_1,\ldots,U_{2m-1}\}$.  But for all odd $j\leq 2m-1$, we have $\dist_{U_j}(x_i,x_{i+1})\leq 10(C+E)$ by Claim \ref{claim:odd-relevant}, so $\relevant_{2K+1}(x_i,x_{i+1})\cap\mathfrak S_{ql}=\emptyset$ because of our choice of $K$.

Finally, we have shown that for each $V\in\mathfrak S$, each $\dist_V(x_i,\gamma_V)\leq 20(E+C)$.  Let $p_i^V$ be a closest point in $\gamma_V$ to $x_i$.  For even $i$, let $\gamma^i_V$ be the sub-geodesic joining $p_i^V,p^{i+1}_V$.  Examining the above possibilities shows that for all even $i$ with $\dist_V(x_i,x_{i_1})\geq 2T$, the geodesic $\gamma^i_V$ has a subpath $\alpha^i_V$ with $|\alpha^i_V|\geq |\dist_V(x_i,x_{i+1})|/2$ and all the $\alpha^i_V$ are disjoint, which proves the claim.
\end{proofofclaim}

Claim \ref{claim:even-i-sum}, Claim \ref{claim:large-link-application}, and the distance formula give:
\begin{eqnarray*}
    \sum_{i\ \text{even}}\dist(x_i,x_{i+1})&\leq & \DF(2K+1)\sum_{i\ \text{even}}\left[\sum_{V\in \mathfrak S-\mathfrak S_{ql}}\ignore{\dist_V(x_i,x_{i+1})}{2K+1}+\DF(2K+1)\right]\\
    &\leq&N\DF(2K+1)^2\epsilon\cdot \dist(x,y)+2\DF(2K+1)\sum_{V\in\mathfrak S-\mathfrak S_{ql}}\ignore{\dist_V(x,y)}{K}\\
    &\leq&N\DF(2K+1)^2\epsilon\cdot \dist(x,y)+4\DF(2K+1)\DF(K)\cdot\hat\dist(x,y)\\
    &\leq&\DF(2K+1)(N\DF(2K+1)+4\DF(K))\cdot\epsilon\cdot\dist(x,y),
\end{eqnarray*}
where we used $\hat\dist(x,y)\geq 1/\epsilon>\DF(K)$ by our choice of $\epsilon_0$.

Next, using Claim \ref{claim:odd-x_i-at-start} and the distance formula, we get:
\begin{eqnarray*}
    \sum_{i\ \text{odd}}\dist(x_i,x_{i+1})&\leq&2\DF(2C+1)\sum_{i\ \text{odd}}\dist_{U_i}(x_i,x_{i+1})\\
    &\leq& 2\DF(2C+1)\sum_{i\ \text{odd}}\left[\dist_{U_i}(x,y)+4C+2E\right]\\
    &\leq& 2\DF(2C+1)\sum_{U\in\relevant_{s_\lambda}(x,y)}\left[\dist_U(x,y)+4C+2E\right]\\
    &\leq&2\DF(2C+1)\DF(s_\lambda,\lambda)\cdot \dist(x,y).
\end{eqnarray*}
where the affine map $\lambda$ is $t\mapsto t+4C+2E$.

\textbf{Conclusion.}  Item \eqref{item:odd-one-relevant} holds by Claim \ref{claim:odd-relevant}.  Next, we choose $M\geq K$ as follows.  First, as long as $M\geq 2\DF(2C+1)\DF(\tau,\lambda)$, the preceding computations give item \eqref{item:odd-sum-linear}.  Second, by taking $M\geq \DF(2K+1)(N\DF(2K+1)+4\DF(K))$, we get item \eqref{item:even-sum-small}.  Finally, as long as $M\geq N$, the uniform constant from Claim \ref{claim:large-link-application}, we get $n\leq M\epsilon\cdot\dist(x,y)$, as required.
\end{proof}

Now we count the points $y$ in a given ball $B^X_r(x)$ whose projections differ significantly from those of $x$ in only one coordinate space $\mathcal CU$, with $U\in\mathfrak S_{ql}$. 

\begin{defn}\label{defn:A-xr}
Let $M\geq 1,r>1$, and let $x\in X$.  Let $A^M_{x,r}$ be the set of $y\in X$ such that:
\begin{itemize}
    \item $\dist(x,y)\leq r$,
    \item $\relevant_M(x,y)\subseteq \mathfrak S_{ql}$,
    \item $|\relevant_M(x,y)|\leq 1$.
\end{itemize}
\end{defn}

The next lemma bounds the size of $A^M_{x,r}$ by a linear function of $r$.  To do this, we divide $A^M_{x,r}$ into two disjoint subsets $\mathcal A,\mathcal B$, with $\mathcal A$ consisting of those $y\in A^M_{x,r}$ that are $2s$--close to $x$, for some suitably chosen $s$, depending only on $M$ and the HHS parameters.  The cardinality of $\mathcal A$ is therefore bounded independently of $r$, so we can ignore $\mathcal A$.  In the first two claims in the proof, we show that any $y\in\mathcal B$ must have the property that $\relevant_M(x,y)\subseteq\{U\}$, where $U$ is one of a uniformly bounded set (independent of $r$) of elements of $\mathfrak S_{ql}$, whose standard product regions are in the immediate vicinity of $x$.  We then argue, using that $\mathcal CU$ is a quasiline, that there are only linearly many possibilities for $y$.

\begin{lem}\label{lem:linear}
    For all $M\geq 1$ there exists $\lambda\geq 1$ such that $$|A^M_{x,r}|\leq \lambda r$$ for all $x\in X$ and $r>1$.
\end{lem}

\begin{proof}
We may assume that $M\geq M_1$, where $M_1$ is as in Lemma \ref{lem:initial-M}, so $\relevant_{2M}(x,y')\subseteq \relevant_M(x,y)$ for all $x,y\in X$ and $y'\in\hull(x,y)$.  We may also assume that $2M\geq s_{Id}$.

Now fix $x\in X$.  Given $s\geq 1$, let $\mathcal Q_s$ be the set of all $U\in\mathfrak S_{ql}$ such that there exists $y\in X$ for which $\dist(x,y)\in [s,2s]$ and $\relevant_{2M}(x,y)=\{U\}$.  

\setcounter{claim}{0}
\begin{claim}\label{claim:bounded-Q}
For any sufficiently large $s$, there exists $N_1\in\naturals$ such that $|\mathcal Q_s|\leq N_1$.
\end{claim}

\begin{proofofclaim}{\ref{claim:bounded-Q}}
Let $s\geq \DF(2M)(4M+1)+\DF(2M)$. Recall that $|\neb_{2s}(x)|\leq \growth(2s)$, which is independent of $x$.  Let $U\in\mathcal Q_s$.  Choose $y\in X$ such that $\relevant_{2M}(x,y)=\{U\}$, which exists by the definition of $\mathcal Q_s$.  Then the distance formula implies that $\dist(x,y)\asymp_{\DF(2M)}\dist_U(x,y)$.  Moreover, the definition of $\mathcal Q_s$ allows us to choose this $y$ with $\dist(x,y)\in[s,2s]$.  Hence $\dist_U(x,y)\geq (s-\DF(2M))/\DF(2M)\geq 4M+1$.

Suppose that $U'\in\mathcal Q_s-\{U\}$ and let $y'\in X$ be such that $\relevant_{2M}(x,y')=\{U'\}$ and $\dist(x,y')\in[s,2s]$. Then $\dist_U(y,y')\geq \dist_U(y,x)-\dist_U(x,y')\geq 2M$, and similarly $\dist_{U'}(y,y')\geq 2M$, so $y\neq y'$. Thus the assignments $U\mapsto y$ define an injective map $\mathcal Q_s\to \neb_{2s}(x)$ and this implies $|\mathcal Q_s|\leq \growth(2s)$.
\end{proofofclaim}

\begin{claim}\label{claim:U-in-Q-s}
For all sufficiently large $s$, the following holds.  Let $s'>2s$ and $U\in\mathfrak S_{ql}$.  Suppose that there exists $y\in X$ such that $\relevant_M(x,y)=\{U\}$ and $\dist(x,y)\geq s'$.  Then $U\in\mathcal Q_s$.
\end{claim}

\begin{proofofclaim}{\ref{claim:U-in-Q-s}}
Let $U\in\mathfrak S_{ql}$ and suppose that there exists $y$ as in the statement.  Let $\gamma:[0,K]\to X$ be a $\kappa$--hierarchy path from $x$ to $y$, where $\kappa$ depends only on the HHS parameters and $\image(\gamma)\subseteq \hull(x,y)$.  Our choice of $M$ ensures that $\relevant_{2M}(x,\gamma(t))\subseteq \relevant_M(x,y)=\{U\}$.  Choose $t'\in[0,K]$ such that $\dist(x,\gamma(t'))\in[3s/2-2\kappa,3s/2+2\kappa]$, which is possible since $\dist(x,y)\geq s'>2s$, and let $y'=\gamma(t')$.  Then, provided $s$ was chosen sufficiently large in terms of $\kappa$ and the distance formula constants (with threshold $2M$), we have $\dist_U(x,y')\geq 2M$ and $\dist(x,y')\in[s,2s]$, so $U\in\mathcal Q_s$, as required.
\end{proofofclaim}

By the preceding two claims, there exist $s\geq 1$ and $N_1\in\naturals$, depending only on the HHS parameters and $M$, such that the following holds.  Let $\mathcal Q$ be the set of $U\in\mathfrak S_{ql}$ such that there exists $y\in X$ with $\dist(x,y)\geq 2s$ and $\relevant_M(x,y)=\{U\}$.  Then $|\mathcal Q|\leq N_1$.  

Moreover, if $y\in X$ satisfies $\relevant_M(x,y)=\emptyset$, then $\dist(x,y)\leq \theta_u(M)$, where $\theta_u$ is the function from the uniqueness axiom, \cite[Defn. 1.1.(9)]{HHS_II}, and we can and shall assume that $s\geq \theta_u(M)$.

To conclude, let $r>0$ be given.  Write
$$A^M_{x,r}=\mathcal A\sqcup\mathcal B,$$
where $\mathcal A=A^M_{x,r}\cap \neb_{2s}(x)$, and $\mathcal B$ therefore consists of all $y\in X$ such that $\relevant_M(x,y)$ has exactly one element, and that element is in $\mathfrak S_{ql}$, and $\dist(x,y)\in[2s+1,r]$.

Observe that $|\mathcal A|\leq \growth(2s)$, which is bounded independently of $r$.  Next, note that if $y\in\mathcal B$, then $\relevant_M(x,y)=\{U\}$ for some $U\in\mathcal Q$, so there are at most $N_1$ possibilities for $U$. Fix $U\in\mathcal Q$ and let $\mathcal B_U=\big\{y\in X:\dist(x,y)\leq r,\ \relevant_M(x,y)=\{U\}\big\}$.  

\begin{claim}\label{claim:linearly-many-per-line}
There exists $\lambda_0$, depending only on the HHS parameters, such that $|\mathcal B_U|\leq \lambda_0r$ for all $U\in\mathcal Q$.
\end{claim}

\begin{proofofclaim}{\ref{claim:linearly-many-per-line}}
Since $(X,\mathfrak S)$ is standard and $U$ is $\nest$--minimal, letting $\bar x$ be the gate of $x$ in $P_U$, the subspace $F_U^{\bar x}$ is a bounded-geometry quasiline, and $\mathcal B_U$ lies in a uniformly bounded neighbourhood of a ball in $F_U^{\bar x}$ of radius bounded uniformly linearly in terms of $r$.  
\end{proofofclaim}

Combining Claim \ref{claim:linearly-many-per-line} with our estimate for $|\mathcal A|$, we get $|A^M_{x,r}|\leq \growth(2s)+N_1\lambda_0r$, and since $r>1$ and $s$ is independent of $r$, we can take $\lambda =\growth(2s)+N_1\lambda_0$ to conclude.
\end{proof}

Now we can prove the main proposition of this section.

\begin{proof}[Proof of Proposition \ref{prop:counting}]
Let $a>1$ be given.  Let $b,c\in (1,a)$ be quantitities to be chosen. Let $M\geq 1,\epsilon_0>0$ be the constants provided by Lemma \ref{lem:sequence}.  For the given $b$, Lemma \ref{lem:easy-count} provides a constant $\epsilon_2>0$ and a constant $r_0>0$ such that 
$$\left|\{v\in T^k_3:d(o,v)\leq r,\ \dist(f(o),f(v))\leq \epsilon_2r\}\right|\leq b^r$$
for all $r>r_0$.

For the given $c>1$, Lemma \ref{lem:partition-counting} provides a constant $\epsilon\in(0,\epsilon_0)$ such that $P(s,M\epsilon s)\leq c^s$ for all $s>0$.  Finally, set $$\epsilon_1=\epsilon\cdot\epsilon_2.$$ We can and shall assume that all $\epsilon_i,\ i\in\{0,1,2\}$, and $\epsilon$, are in $(0,1)$. Let $r>r_0$ and let 
$$B_{o,r}=\{v\in T_3^k: d(o,v)\leq r,\  \hat\dist(\hat f(o),\hat f(v))<\epsilon_1 r,\  \dist(f(o), f(v))\geq \epsilon_2 r\},$$
and let 
$$B'_{o,r}=\{v\in T_3^k: d(o,v)\leq r,\  \hat\dist(\hat f(o),\hat f(v))<\epsilon_1 r,\  \dist(f(o), f(v))<\epsilon_2 r\},$$
so that our ultimate goal is to show that $|B_{o,r}\sqcup B'_{o,r}|\leq a^r$ for sufficiently large $r$, provided $\epsilon$, and hence $\epsilon_1=\epsilon\cdot\epsilon_2$, is sufficiently small. First, $|B'_{o,r}|\leq b^r$ for all $r>r_0$, by Lemma \ref{lem:easy-count} and our choice of $\epsilon_2,r_0$.

\textbf{Bound on $|f(B_{o,r})|$.} Set $x=f(o)$.  Note that if $p\in B_{o,r}$ and $f(p)=:y$, then $\epsilon_2r\leq \dist(x,y)\leq Lr$, and $\hat\dist(x,y)<\epsilon_1 r=\epsilon\cdot\epsilon_2r$, so Lemma \ref{lem:sequence} applies to $x,y$.  The lemma therefore provides $n\leq M\epsilon\dist(x,y)\leq M\epsilon Lr$ and a sequence of points $x=x_0,x_1,\cdots,x_n=y$ in $X$ with the properties enumerated in the lemma.  Let $s=\sum_{i=0}^{n-1}\dist(x_i,x_{i+1})$, so that $\dist(x,y)\leq s\leq M(1+\epsilon)Lr$.  Then the quantities $\dist(x_i,x_{i+1})$ partition $s$ into $n$ parts.

Now fix $i$ and suppose that the point $x_i$ is given.  Let $\eta_i$ be the number of possible choices for $x_{i+1}$, given $x_i$, that satisfy the conclusions of Lemma \ref{lem:sequence}. Then:
\begin{itemize}
    \item if $i$ is even, then $\eta_i\leq \growth(\dist(x_i,x_{i+1}))\leq a_0^{\dist(x_i,x_{i+1})}$, and
    \item if $i$ is odd, then $\eta_i\leq\lambda \dist(x_i,x_{i+1})$, where $\lambda$ is the constant from Lemma \ref{lem:linear} (here we used Lemma \ref{lem:sequence}.\eqref{item:odd-one-relevant}).
\end{itemize}
Note that $n\leq M\epsilon \dist(x,y)\leq M\epsilon s$. Hence, using Lemma \ref{lem:sequence} and the arithmetic/geometric mean inequality, 
\begin{eqnarray*}
\prod_{i=0}^{n-1}\eta_i&\leq&a_0^{\sum_{i\ \text{even}}\dist(x_i,x_{i+1})}\cdot\prod_{i\ \text{odd}}\eta_i\\
&\leq&a_0^{M\epsilon Lr}\cdot \lambda^{n/2}\left(\frac{2MLr}{n}\right)^n\leq a_0^{M\epsilon Lr}\cdot\lambda^{M\epsilon s}\left(\frac{2}{\epsilon}\right)^{\epsilon MLr}\cdot e^{\epsilon MLr/e}.
\end{eqnarray*}
Indeed, $$\left(\frac{2MLr}{n}\right)^n\leq \left(\frac{2}{\epsilon}\right)^{\epsilon MLr}\cdot \max_{z\geq 1}\left(\frac{\epsilon MLr}{z}\right)^z.$$
Hence, letting $a_1=a_0e^{1/e}$, the total number of possibilities for $y$ is bounded above by 
$$a_1^{M\epsilon Lr}\cdot\left(\frac{2}{\epsilon}\right)^{\epsilon MLr}\cdot\sum_{s\leq M(1+\epsilon)Lr}P(s,M\epsilon s)\lambda^{M\epsilon s}\leq \left[\left(\frac{2a_1}{\epsilon}\right)^{ML\epsilon}\right]^r\cdot\sum_{s\leq M(1+\epsilon)Lr}(c\lambda^{M\epsilon})^s.$$
Hence, for any $b_1>1$, by choosing $c$ and $\epsilon$ sufficiently small, we can ensure that there exists $r_1$ such that $|f(B_{o,r})|\leq b_1^r$ for all $r>r_1$.   

Recall that there is a constant $N_0$ such that $|f^{-1}(x)|\leq N_0$ for all $x\in X$.  So, for $r>r_1$, we have $|B_{o,r}|\leq N_0b_1^r$.  Hence $|B_{o,r}\cup B'_{o,r}|\leq b^r+N_0b_1^r$ for $r>\max\{r_0,r_1\}$, and since this can be arranged for arbitrary $b,b_1$ by making $\epsilon$ sufficiently small and $r_0,r_1$ sufficiently large, we get $|B_{o,r}\cup B'_{o,r}|\leq a^r$, as required.
\end{proof}

\section{Proof of Theorem \ref{thm:main}}\label{sec:induction}
The goal of this section is to prove Theorem \ref{thm:main}, which we now recall.

\main*

\subsection{Arcs in fibres}\label{subsec:gmt-facts}
We will need some general facts, assembled from the results in \cite{Simon:GMT} and \cite{EilenbergHarrold:continua}. Throughout, if $f:\reals^n\to\reals^m$ is a lipschitz map, given by $f((x_i)_{i=1}^n)=(f_j((x_i)_{i=1}^n))_{j=1}^m$, then Rademacher's theorem (see e.g. \cite[Thm. 2.1.4]{Simon:GMT}) implies that the gradients $\nabla f_j$ are defined at almost every point in $\reals^n$, so we denote by $\jacobian f$ the Jacobian of $f$, i.e. the matrix with $(i,j)$ entry $\frac{\partial f_j}{\partial x_i}$.  We also let $\leb^n$ denote the Lebesgue measure on $\reals^n$.  We denote by $\haus^n$ the $n$--dimensional Hausdorff measure on subsets of whichever metric space we are working in.  We refer to embedded (rectifiable) paths in metric spaces as \emph{(rectifiable) arcs}. 

The next several lemmas are to support Corollary \ref{cor:arc}, which is one of the main ingredients in the proof of the theorem.  The input to the corollary is a lipschitz map $f:\reals^k\times [-1,1]\to \reals^k$ with suitable properties, and the output is a rectifiable arc, in the preimage of a point, that joins $\reals^k\times \{\pm 1\}$ and on which the Jacobian of $f$ is almost always of full rank.

\begin{lem}
\label{lem:geometric-analysis-magic}
    Let $f:\mathbb R^k\times [-1,1]\to \mathbb R^k$ be a proper lipschitz function. Then for $\haus^k$-almost every point $x$ of $\mathbb R^k$, the set $f^{-1}(x)$ has finite $\haus^1$-measure, and at
    $\haus^1$-almost-every point of $f^{-1}(x)$, the matrix $\jacobian f$ has rank $k$.
\end{lem}

\begin{proof}
Let $A=\reals^k\times[-1,1]\subseteq \reals^{k+1}$, and let $g:\reals^{k+1}\to \reals^k$ be a lipschitz map with $g|_A=f$.  Let $C=\{x\in\reals^{k+1}:\rank(\jacobian{g}(x))<k\}$.  For each $x\in\reals^{k+1}$, let $\Delta(x)=\sqrt{\det\left(\jacobian{g}(x)\jacobian{g}(x)^t\right)}$.  Then the co-area formula (\cite[Thm. 2.7.3]{Simon:GMT}) says that
$$\int_{\reals^k}\haus^1(B \cap f^{-1}(y))\ d\leb^k(y)=\int_{B}\Delta(x)\ d\haus^{k+1}(x)$$
for any $\haus^{k+1}$--measurable $B\subseteq A$.

Considering the case where $B\subset A$ is an open ball in $\reals^{k+1}$ shows that $\haus^1(B\cap f^{-1}(y))<\infty$ for $\leb^k$--a.e. $y\in\reals^k$, and hence for $\haus^k$--almost all $y$.  Since $f$ is proper, it follows that $f^{-1}(y)$ has finite $\haus^1$--measure for almost all $y$.

Next, consider the case where $B=C\cap A$.  Now, $x\in B$ if and only if $x\in A$ and $\rank(J)<k$, where $J=\jacobian{f}(x)$.  Since $\rank(J)=\rank(JJ^t)$, we have $x\in B$ if and only if $\Delta(x)=0$.  Hence, for $\leb^k$--almost every $y$, we have $\haus^1(C\cap f^{-1}(y))=0$, which completes the proof.
\end{proof}

\begin{lem}
\label{lem:path-top-to-bottom}
    Let $h:\mathbb R^k\times [-1,1]\to\mathbb R^k$ be a proper lipschitz map such that $h|_{\mathbb R^k\times \{\delta\}}$ is bilipschitz for some $\delta\in[-1,1]$. Then for all $p\in\mathbb R^k$ there exists a connected subset $\alpha'$ of $h^{-1}(p)$ that intersects both $\mathbb R^k\times \{-1\}$ and $\mathbb R^k\times \{1\}$.
\end{lem}

\begin{proof}
Fix $p\in\reals^k$ and let $A=f^{-1}(p)$. It will be useful several times below that, since $h$ is proper, $A$ is compact.  Let $R^-=\reals^k\times \{-1\}$ and $R^+=\reals^k\times \{+1\}$. If some connected component of $A$ intersects both $R^-$ and $R^+$, then we are done.  Therefore, suppose that every connected component of $A$ is disjoint from either $R^-$ or $R^+$. 

\setcounter{claim}{0}

\begin{claim}\label{claim:component}
 There exists a clopen $C\subset R^-\cup A\cup R^+$ with $R^-\subseteq C$ and $C\cap R^+=\emptyset$.
\end{claim}

\begin{proof}
The connected component $C'$ of $R^-\cup A\cup R^+$ that contains $R^-$ is disjoint from $R^+$ because we are assuming that no connected component of $A$ intersects both $R^-$ and $R^+$. In compact Hausdorff spaces it is well-known that connected components are intersections of clopen sets. Therefore, for any closed ball $B$ in $\reals^k\times [-1,1]$, there exists a clopen set $D$ of $B\cap(R^-\cup A\cup R^+)$ containing $C'\cap B$. Take $B$ to be a ball that contains the compact set $A$ and note that $D\cup R^-$ is clopen, and  set $C=D\cup R^-$.
\end{proof}

 Fix $C$ as in Claim \ref{claim:component}.  The goal of the next claim is to produce a codimension--$1$ submanifold $M\subset \reals^k\times [-1,1]$ which agrees with $\reals^k\times \{\star\}$ outside of some ball, and is disjoint from $A$.  This submanifold will be a level set of the function $\lambda$ from the claim.
 
\begin{claim}\label{claim:smooth}
    There exists a ball $B_0$ in $\reals^k\times[-1,1]$ and a smooth function $\lambda:\reals^k\times[-1,1]\to \reals$ such that $\lambda(x,t)=t+1$ on $\reals^k\times[-1,1]-B_0$, and $\lambda^{-1}(0)=C$.
\end{claim}

\begin{proof}
    Let $B_2\subset B_1\subset B_0$ be open balls $\reals^k\times[-1,1]$ containing $A$, so that $C\cap B_2^c=R^-\cap B_2^c$. Using a bump function, we can construct a smooth function $\lambda_1:\reals^k\times[-1,1]\to[0,2]$ such that $\lambda_1=0$ on $B_2$ and $\lambda_1(x,t)=t+1$ on the complement of $B_1$. Consider also a smooth function $\lambda_2:\reals^k\times[-1,1]\to[0,1]$ such that $\lambda_2^{-1}(0)=C\cup B_0^c$. Then we can take $\lambda=\lambda_1+\lambda_2$, since $\lambda=\lambda_2$ on $B_2$, $\lambda(x,t)=\lambda_1(x,t)=t+1$ on the complement of $B_0$, and on $B_0-B_2$ we have that $\lambda$ is the sum of two non-negative functions that both vanish exactly on $(B_0-B_2)\cap C= (B_0-B_2)\cap R^-$.
\end{proof}

Note that there exists $\epsilon_0>0$ such that $\lambda(x,t)>\epsilon_0$ for all $(x,t)\in A-C$, using that $A$ is compact and $C\cap A$ clopen in $A$. Consider a regular value $t_0\in (0,\epsilon_0)$ for $\lambda$. Then $M:=\lambda^{-1}(t_0)$ has the following properties:
\begin{itemize}
\item $M$ is a manifold.
    \item $M\cap A=\emptyset$, so $h|_M$ is not surjective.
    \item $M-B_0=\reals^k\times \{t_0-1\}-B_0$.

\end{itemize}

We can extend $h:M\to \reals^k$ to a map $\bar h$ between the respective one-point compactifications $\bar M$ and $\bar\reals^k\cong S^k$, which are closed $k$-manifolds, and therefore have non-trivial $k$-th homology with $\mathbb Z/2$-coefficients. We consider the following commutative diagram involving local homologies at the compactification point, with $\mathbb{Z}/2$-coefficients.

\begin{center}
\begin{tikzcd}
H_k(\bar M) \arrow[r]\arrow[d] & H_k(\bar M,M)\arrow[d]\\
H_k(\bar \reals^k) \arrow[r] & H_k(\bar\reals^k,\reals^k)
\end{tikzcd}
\end{center}

The horizontal maps come from long exact sequences for pairs, and they are surjective since $\bar M$ and $\bar \reals^k$ are closed manifolds. The vertical arrows are induced by $\bar h$.

Since $\bar h$ is not surjective, it induces the trivial map at the level of $H_k$, which forces the right vertical map to be the trivial. However, excision and the topological Lemma \ref{lem:local_hom} below show that it is injective, a contradiction.
\end{proof}

\begin{lem}
\label{lem:local_hom}
    Let $X$ be the complements in $\reals^k$ of an open ball $B$, and let $f:X\to \reals^k$ be a continuous quasi-isometric embedding. Let $\bar X,\bar \reals^k$ be the one-point compactifications, obtained adding the point $p$, and let $\bar f$ be the extension of $f$ with $\bar f(p)=p$.  Then $\bar f_*:H_k(\bar X, X)\to H_k(\bar \reals^k, \reals^k)$ (where we use $\mathbb Z/2$ coefficients) is an isomorphism.
\end{lem}

\begin{proof}
The map $f$ is a quasi-isometry by \cite[Lem. 10.84]{DrutuKapovich:book}, so we can fix a quasi-inverse $g$. Choose some ball $B'\supset B$, such that $0\notin f(\reals^k-B')$, and large enough to be determined by the following argument. 

There is a chain $c$ representing a generator of $[c]\in H_k(\bar X,X)$, whose boundary $\partial c=\sum \sigma_i$ is supported on $\partial B'$, with the image of each $\sigma_i$ having diameter at most $1$. For $B''$ a  ball in $\reals^k$, there are isomorphisms 
\begin{center}
\begin{tikzcd}
H_k(\bar \reals^k,\reals^k) &\arrow[l] H_k(\bar \reals^k -B'',\reals^k -B'') \arrow[r] & H_{k-1}(\reals^k-B'')
\end{tikzcd}
\end{center}

\noindent respectively given by excision and by the connecting homomorphism (the latter is an isomorphism since $\bar \reals^k -B''$ is contractible). 

In view of this, and by construction of the connecting homomorphism, $\bar f_*([c])$ represents a non-trivial element of $H_k(\bar \reals^k,\reals^k)$ provided that $\bar f_*([\partial c])$ is a non-trivial element of $H_{k-1}(\reals^k-B'')$ for some ball $B''$ containing $0$ and not intersecting $f(\boundary B')$. 

Suppose instead $\bar f_*([\partial c])=[\partial d]$ for some chain $d$ supported outside a sufficiently large ball $B''$. Using barycentric subdivisions, we can assume that $d$ consists of singular simplices whose images have diameter at most $2$. We can form a chain $d'$ in $\reals^k$ with $\partial d'=\partial c$ by coarsely pulling back $d$ under the quasi-isometry $g$. More precisely, we can consider singular simplices which are affine maps, and whose vertices each get mapped to either the image of a vertex of some $\sigma_i$, or of the form $g(\sigma(v))$ for some vertex $v$ of a singular simplex $\sigma$ of $d$. If $B''$ is sufficiently large, we can arrange $d'$ to be supported outside $B$. But this is a contradiction because $\partial c$ does not represent the trivial class in $H_{k-1}(X)$ because the connecting homomorphism $H_k(\bar X, X)\to H_{k-1}(X)$ is an isomorphism.
\end{proof}

We now state and prove the last remaining lemma needed to produce our arc.

\begin{lem}
\label{lem:arc-in-preimage}
    Let $X\subseteq \reals^k\times[-1,1]$ be a compact connected set with $\haus^1(X)<\infty$. Then any two distinct points of $X$ are connected by a nontrivial rectifiable arc in $X$.
\end{lem}

\begin{proof}
By \cite[Thm. 2]{EilenbergHarrold:continua}, there is a  rectifiable path in $X$ joining the given points.  By \cite[Prop. 2.5.19]{BBI}, there is a shortest such rectifiable path, which must be an arc.
\end{proof}

\begin{cor}\label{cor:arc}
Let $f:\reals^k\times [-1,1]\to\reals^k$ be a proper lipschitz map such that $f|_{\reals^k\times\{\delta\}}$ is bilipschitz for some $\delta$.  Then for $\haus^k$--almost all $p\in \reals^k$, there is a rectifiable arc $\alpha:[-1,1]\to \reals^k\times[-1,1]$ such that
\begin{itemize}
    \item $f\circ \alpha$ is the constant map to $p$;
    \item $\alpha^{-1}(\reals^k\times\{1\})=\{1\}$ and $\alpha^{-1}(\reals^k\times\{-1\})=\{-1\}$;
    \item for $\haus^1$--almost all $t\in[-1,1]$, the Jacobian $\jacobian{f}(\alpha(t))$ has rank $k$.
\end{itemize}
\end{cor}

\begin{proof}
This follows from Lemma \ref{lem:geometric-analysis-magic}, Lemma \ref{lem:path-top-to-bottom}, and Lemma \ref{lem:arc-in-preimage}.
\end{proof}

We will also need the following criterion for a lipschitz map $f:\reals^k\times [-1,1]\to \reals^{k+1}$ to admit a point where the Jacobian full rank. Roughly, the conditions required are that there exist a hyperplane and a curve that crosses it at a point where the Jacobians of components of $f$ have maximal rank.  This lemma is used to prove Corollary \ref{cor:full_rank}, where $f$ comes from the gate map to a standard box.

\begin{lemma}
\label{lem:full_rank_h}
    Let $h:\reals^k\times [-1,1]\to \reals^{k+1}$ be a lipschitz map with the following properties, where $\pi_1:\reals^{k+1}\to \reals^k$ is the projection to the first $k$ factors, $\pi_2:\reals^{k+1}\to \reals$ is the projection to the last factor, and we denote by $0_k$ and $0_{k+1}$ the origins of $\reals^k$ and $\reals^{k+1}$.
\begin{enumerate}
    \item $\jacobian{\pi_1\circ h}$ has rank $k$ at $0_k$.
    \item There is a lipschitz curve $\alpha:[-1,1]\to \reals^k\times [-1,1]$ with $\alpha(0)=0_{k+1}$ and such that $\jacobian{\pi_2\circ h\circ \alpha}$ has rank 1.
    \item $\pi_1\circ h\circ \alpha$ is a constant map.
\end{enumerate}
Then $\jacobian{h}$ has rank $k+1$ at some point of $\reals^k\times [-1,1]$.
\end{lemma}

\begin{proof}
Let $D=\reals^k\times [-1,1]$. Since $h$ is lipschitz, Rademacher's theorem \cite[Thm. 1.4]{Simon:GMT} implies that $\jacobian{h}$ is defined $\leb^{k+1}$--a.e. in $D$, and then \cite[Thm. 3.3]{Simon:GMT} says that $\int_{D}|\jacobian{h}|d\leb^{k+1}\geq \haus^{k+1}(h(D))$, so it suffices to show that $h(D)$ has nonempty interior.

Let $h_i=\pi_i\circ h$ and $\alpha_i=\pi_i\circ \alpha$,  for $i=1,2$. Note that $h(x)=(h_1(x),h_2(x))$ for all $x\in D$, and similarly $\alpha=(\alpha_1,\alpha_2)$. By (1), up to changing coordinates on $D$ and $\reals^{k+1}$ we have $\dist(h_1(b,t),b)=o(\max\{|b|,|t|\})$.

Let $\sigma$ be the map $\reals^{k+1}=\reals^k\times\reals$ to itself given by $(x,t)\mapsto (x,-t)$. Given a ball $B$ and $\epsilon>0$, we denote by $\epsilon B$ the ball of the same centre and radius $\epsilon r$ if the radius of $B$ is $r$.

\begin{claim}\label{claim:straight_line}
There exists $\epsilon_0>0$ such that for all $M>0$ there exists $r_0>0$ such that for any ball $B\subseteq \reals^k$ centred at $\vec 0$ of radius $r<r_0$ the following holds. For all $u\in(\partial B)\times [-Mr,Mr]$ the line segment in $\reals^{k+1}$ from $u$ to $h(u)$ is disjoint from $\epsilon_0 B\times [-\epsilon_0Mr,\epsilon_0Mr]$, as is the one from $\sigma(u)$ to $h(u)$.
\end{claim}

\begin{proofofclaim}{\ref{claim:straight_line}}
    This is because this holds even for the line segment joining $\pi_1(u)=\pi_1(\sigma(u))$ to $\pi_1(h(u))=h_1(u)$ since these two points are much closer to each other than they are to $0_k$.
\end{proofofclaim}

\begin{claim}\label{claim:signs}
    There exists $\epsilon>0$ such that for all sufficiently small $t>0$ we have one of the following:

    \begin{itemize}
        \item $h_2(0,t)>\epsilon t$ and $h_2(0,-t)<-\epsilon t$, or
        \item $h_2(0,t)<-\epsilon t$ and $h_2(0,-t)>\epsilon t$.
    \end{itemize}
     
\end{claim}

\begin{proofofclaim}{\ref{claim:signs}}
In this proof, we say that a function $f(t)$ is $\Theta(t)$ if there exists $\epsilon>0$ such that whenever $|t|$ is sufficiently small, we have $|\Theta(t)|>\epsilon|t|$.

    Since $\dist(h_1(b,t),b)=o(\max\{|b|,|t|\})$ and $\alpha$ is lipschitz, we also have $\dist(h_1(\alpha(t)),\alpha_1(t))=o(|t|)$. Since $h_1\circ\alpha$ is constant, and in fact in our coordinates constant equal to $0_k$, we have $|\alpha_1(t)|=o(t)$ (that is, $\alpha$ stays sublinearly close to $\{0_k\}\times [-1,1]$).

Since $\jacobian{h_2\circ \alpha}$ has rank 1 for the lipschitz function $h_2$, we have $|h_2\circ\alpha(t)|=\Theta(t)$, which since $h_2$ is lipschitz also implies $|\alpha(t)|=\Theta(|t|)$, and in turn we must also have $|\alpha_2(t)|=\Theta(t)$ since $|\alpha_1(t)|=o(t)$.

Consider the two curves $\alpha^{+}=\alpha|_{[0,1]}$ and $\alpha^{-}$ which we take to be the inverse of $\alpha|_{[-1,0]}$. Due to $|\alpha_2(t)|=\Theta(t)$ and $|\alpha_1(t)|=o(t)$, for each sufficiently small $t>0$ we have that there exists $s^+=\Theta(t)$ with $d(\alpha^+(s^+),(0_k,t))=o(t)$ or $d(\alpha^+(s^+),(0_k,-t))=o(t)$, and similarly for $\alpha^-$. But $\alpha^+(s^+)$ and $\alpha^-(s^-)$ cannot be $o(t)$-close because, after applying the lipschitz function $h_2$, these get mapped to points $\Theta(t)$ away from each other since $\jacobian{h_2\circ\alpha}$ has rank 1. Therefore, keeping into account that $h_2$ is lipschitz, we have that $h_2(0,t)$ is $o(t)$-close to, say, $h_2(\alpha^+(s^+))$, and then $h_2(0,t)$ is $o(t)$-close to $h_2(\alpha^-(s^-))$. Since $h_2\circ\alpha$ has nonzero differential and $s^{\pm}$ are $\Theta(t)$, we get the required conclusion.
\end{proofofclaim}

Figure \ref{fig:cyl} illustrates the following claim.

\begin{claim}\label{claim:straight_line_plus}
There exists $\epsilon>0$, a ball $B\subseteq \reals^k$ centred at $0_k$, and $s>0$ such that one of the following holds.
\begin{enumerate}
    \item For all $u\in\partial (B\times [-s,s])$ the line segment in $\reals^{k+1}$ from $u$ to $h(u)$ is disjoint from $\epsilon B\times [-\epsilon s,\epsilon s]$.
    \item For all $u\in\partial (B\times [-s,s])$ the line segment in $\reals^{k+1}$ from $\sigma(u)$ to $h(u)$ is disjoint from $\epsilon B\times [-\epsilon s,\epsilon s]$.
\end{enumerate}
\end{claim}

\begin{proofofclaim}{\ref{claim:straight_line_plus}}
    Let $\epsilon>0$ as in Claim \ref{claim:signs}. Since $h_2$ is lipschitz, there exists $M$ such that for $t>0$
    \begin{enumerate}
        \item if the first case of Claim \ref{claim:signs} holds for $t$, then we have $h_2(b,t)>\epsilon t/2$ and $h_2(b,-t)<\epsilon t/2$ whenever $|b|<t/M$,

    \item if instead the second case holds for $t$, then $h_2(b,t)<-\epsilon t/2$ and $h_2(b,-t)>\epsilon t/2$ whenever $|b|<|t|/M$.
    \end{enumerate}

     In particular, in the first case for such $b$ and $t$ the line segment from $h(b,t)$ to $(b,t)$  does not intersect $0_{k+1}$ (just by looking at the last coordinate), while in the second case this holds for the line segment from from $h(b,t)$ to $(b,-t)=\sigma(b,t)$.

     By Claim \ref{claim:straight_line}, we can then choose a ball of some radius $r$ and set $s=Mr$ such that the conclusion of the claim holds.
\end{proofofclaim}

\begin{figure}[h]
\medskip
   \begin{overpic}[width=0.5\linewidth]{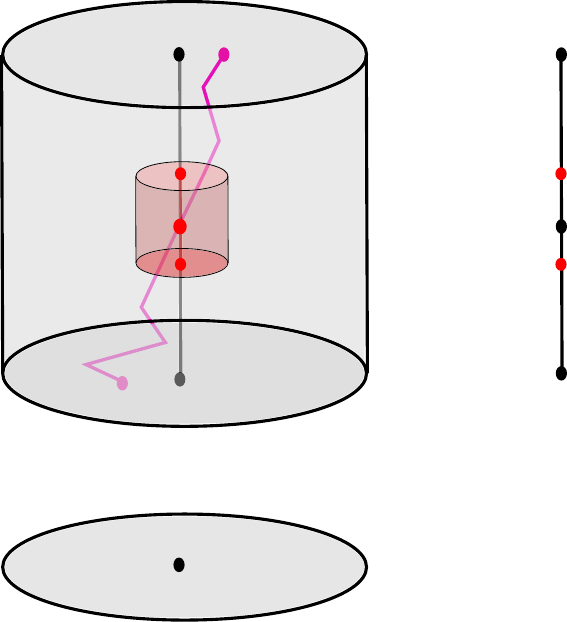}
   \put(70,62){$\stackrel{\pi_2}
   {\longrightarrow}$}
   \put(28,24){$\big\downarrow\ \pi_1$}
   \put(86,90){$s$}
   \put(84,71){$\epsilon s$}
   \put(86,63){$0$}
   \put(81,55){$-\epsilon s$}
   \put(83,40){$-s$}
   \put(31,8){$0_k$}
   \put(31,62){$0_{k+1}$}
   \put(10,40){$\alpha$}
   \put(36,77){$\alpha$}
   \put(50,8){$B$}
   \end{overpic}
   \medskip
    \caption{The cylinder $B\times [-s,s]$ produced in Claim \ref{claim:straight_line_plus}.}
    \label{fig:cyl}
\end{figure}
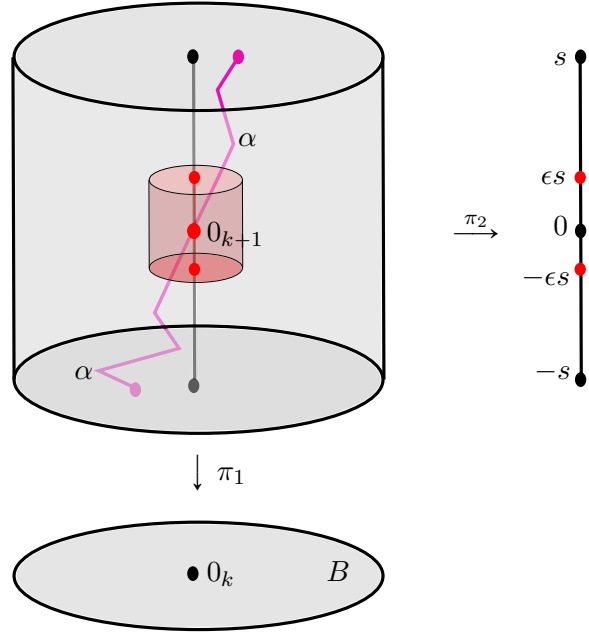

\textbf{Conclusion.} Let $\epsilon$ and $N=B\times [-s,s]$ be as in Claim \ref{claim:straight_line_plus}. We claim that $\epsilon B\times [-\epsilon s,\epsilon s]$ is contained in $h(D)$.

Let $p\in \epsilon B\times [-\epsilon s,\epsilon s]$, and let us argue that $p$ in fact lies in $h(N)$. We have that $\partial N$ represents a non-trivial element of $H_k(\reals^{k+1}-\{p\})$, as does $\partial N=\sigma(\partial N)$ (where we fix a triangulation of $\partial N$ in order to regard it as a cycle). By Claim \ref{claim:straight_line_plus}, $h(\partial N)$ is homotopic to $\partial N$ in $\reals^{k+1}-\{p\}$, so it also represents a non-trivial element of $H_k(\reals^{k+1}-\{p\})$. This forces $h(N)$ to contain $p$, as required.
\end{proof}

\subsection{Inductive setup}\label{subsec:inductive-setup}
We will prove Theorem \ref{thm:main} by induction on $k$.  In the base case, $k=0$, and the statement is immediate.  Therefore, suppose that $k\geq 0$ and that the theorem holds for $k$.  Fix a  coarse embedding $f:T_3^{k+1}\to X$.  

We retain the notation from Section \ref{subsec:counting} so, for instance, $d$ is the $\ell_\infty$ metric on $T_3^{k+1}$ and on $T_3^k$.  We also often use the convention that $T_3^{k+1}$ is identified with $T_3^k\times T_3$, so points in $T_3^{k+1}$ are denoted by ordered pairs $(v,p)$ where $v\in T_3^k$ and $p\in T_3$.

Fix a basepoint $o\in T_3$ and apply the theorem to the restriction $f_k$ of $f$ to $T_3^k\times \{o\}$.  This yields the following data:
\begin{itemize}
    \item An asymptotic cone $(\hatcone,\cdist)$ of $(\widehat X,\hat\dist)$, constructed using an ultrafilter $\omega$, a rescaling sequence $(r_n)$, and a sequence $(o_n)_n$ of observation points.

    \item For $1\leq i\leq k$, a pair of sequences $(a_i^n)_n,(b_i^n)_n$ of vertices in $T_3$, yielding a sequence of cuboids $C_n=\prod_{i=1}^k[a_i^n,b_i^n]$ in $T_3^k$ and hence maps $\hat f|_{C_n}:C_n\times\{o\}\to \widehat X$. 
\end{itemize}
Let $\ctreedist=\omega-\lim_n d/r_n$, which is the metric on $\tcone:=\omega-\lim_n T_3^{k+1}$.  Let $\mathbf C=\omega-\lim_n C_n\times\{o\}\subseteq \tcone$ be the rescaled ultralimit and let $\fcone:\mathbf C\to \hatcone$ be the ultralimit of the $\hat f_n|_{C_n\times\{o\}}$.

The inductive hypothesis is that there is a standard $k$--flat $\mathbf F\subseteq \hatcone$ whose gate map $\gate:\hatcone\to\mathbf F$ has the property that $\gate\circ\fcone:\mathbf C\to\mathbf F$ is $L_0$--bilipschitz for some $L_0\geq 1$.  Since $\gate:\hatcone\to\mathbf F$ is uniformly lipschitz by Lemma \ref{lem:hedge-gate-support}, we can choose $L_0$ so that $\fcone:\mathbf C\to \hatcone$ is an $L_0$--bilipschitz embedding.

Note that $\hat f:T_3^{k+1}\to \widehat X$ limits to a map $\fcone:\tcone\to \hatcone$ whose restriction to $\omega-\lim_nC_n\times\{o\}$ is the above map $\fcone:\mathbf C\to\hatcone$.  The map $\fcone$ is lipschitz, since $f$ is a coarse embedding, and we can enlarge $L_0$ as above so that $\hat f$ is $L_0$--lipschitz on $\tcone$.  

Since $\mathbf F$ is a standard $k$--flat, we also have the following data:
\begin{itemize}
    \item The support of $\mathbf F$ (Definition \ref{defn:support-standard-flat}) is $\{\mathbf U_1,\ldots,\mathbf U_k\}$, where $\bU_i=\omega-\lim_n U_i^n$, where, for all $i$, we have that $U_i^n\in \mathfrak S-\mathfrak S_{ql}$ for $\omega$--a.e. $n$.

    \item For each $i$, we have a valid sequence $(\gamma_i^n)_n$ of hierarchy intervals in $\widehat X$, such that the following holds for $\omega$--a.e. $n$: there is a uniformly quasimedian uniform quasi-isometric embedding $\gamma^n:\prod_{i=1}^k\gamma_i^n\to \widehat X$ whose image $H_n$ is a hierarchy box for $\{U_1^n,\ldots,U_k^n\}$ (see Definition \ref{defn:hierarchy-box}).  Moreover, the sequence $(H_n)_n$ is a valid sequence of hierarchy boxes (Definition \ref{defn:valid-boxes}).
    
    \item $\omega-\lim_n H_n=\mathbf F$.
\end{itemize}
We let $\gate_n:\widehat X\to H_n$ be the (coarse) gate map.  Recall that $\gate:\hatcone\to \mathbf F$ coincides with $\omega-\lim_n\gate_n$, by Lemma \ref{lem:hedge-gate-support}, so $\gate\circ \fcone$ is the ultralimit of the maps $\gate_n\circ \hat f$.

\subsection{The inductive step}\label{subsec:inductive-step}
The first step is to find a point $\mathbf p=(p_n)_n\in \tcone$ such that the boxes $\prod_{i=1}^k[a_i^n,b_i^n]\times \{p_n\}$ limit to a subspace $\mathbf C_{\mathbf p}$ of $\tcone$ whose image under $\fcone$ is disjoint from that of $\mathbf C$.  We will then use Corollary \ref{cor:arc} to join the slices $\mathbf C_{\mathbf o}$ and $\mathbf C_{\mathbf p}$ by an arc $\alpha$ such that moving along $\alpha$ does not change the gate to $\mathbf F$.  This is where we use the estimate from Proposition \ref{prop:counting}.

\begin{lem}\label{lem:good_direction}
There exists a sequence $(p_n)_n$ in $T_3$ with the following properties:
\begin{enumerate}
    \item $d(o,p_n)=r_n$ for $\omega$--a.e. $n$.

    \item Let $\bC_\bp=\omega-\lim_n C_n\times\{p_n\}$ and let $\bC_\bo=\omega-\lim_nC_n\times\{o\}$. Then $$\fcone(\bC_\bo)\cap\fcone(\bC_\bp)=\emptyset.$$
\end{enumerate}    
\end{lem}

\begin{proof}
Let $L=100L_0$.  Let $\delta>1$ be a constant to be chosen momentarily.  Choose $a>1$ sufficiently small that $a^{L^2+2}<\delta$.  For this value of $a$, Proposition \ref{prop:counting} provides a constant $\epsilon>0$ with the property that 
$$|\{(v,w)\in T_3^{k+1}:d((v,w),(b,o))\leq r,\ \hat\dist(\hat f((v,w)),\hat f((b,o))<\epsilon r\}|<a^r$$
for all sufficiently large $r$, where $b\in T_3^k$ is a fixed basepoint.  Noting that $\omega-\lim_n r_n/d(a^n_i,b_i^n)=0$ for each $i$, fix a sequence $(j_n)_n$ of positive integers such that for all $i\leq k$ we have $$r_n\ll j_nr_n\ll d(a_i^n,b_i^n)$$ and $j_n\leq r_n$ for $\omega$--a.e. $n$.

For each $n$, let $N_n$ be the vertex set of $C_n\cap \neb_{j_nr_n}(b)\subseteq T_3^k$, where $T_3^k$ is given the natural product cell structure; all we need from $N_n$ is that it is a $1$--separated net. Since $j_n\leq r_n$, there is a fixed polynomial $\beta:\reals_+\to\reals_+$ such that $|N_n|\leq \beta(r_n)$ for $\omega$--a.e. $n$.  

Given $x,y\in N_n$ for which $d(x,y)\leq (L^2+2)r_n$, let 
$$A_{x,y}=\left\{p\in \neb^{T_3}_{r_n}(o):\hat d( \hat f(x,o),\hat f(y,p))< \epsilon r_n\right\}.$$
Since $d((x,o),(y,p))\leq (L^2+2)r_n$, we can apply Proposition \ref{prop:counting} with $r=(L^2+2)r_n$ and hence
\begin{equation}\label{eqn:A-count}\tag{$\star$}
|A_{x,y}|\leq a^{(L^2+2)r_n}.
\end{equation}

Let $\mathcal P=\{(x,y)\in N_n^2:d(x,y)\leq (L^2+2)r_n\}$ and let 
$A=\bigcup_{(x,y)\in\mathcal P}A_{x,y}$, which is to say that $A$ is the set of all $p\in T_3$ such that $\hat d( \hat f(x,o),\hat f(y,p))< \epsilon r_n$ for some $x,y\in N_n$ with $d(x,y)\leq (L^2+2)r_n$.  So, from the definition of $\beta$ and the estimate \eqref{eqn:A-count}, we get $|A|\leq \beta(r_n)^2\delta^{r_n}$.  Since $\beta$ depends only on the input data and our choice of $(j_n)_n$, which can be chosen in terms of the input data only, we can assume that $\delta$ is sufficiently small that $\beta(r_n)^2\delta^{r_n}<3\cdot 2^{r_n}$ for $\omega$--a.e. $n$.  So, for $\omega$--a.e. $n$, there is a point $p_n\in T_3$ such that $p_n\not \in A$ and $d(o,p_n)=r_n$.  Let $\bp=\omega-\lim_n p_n$. 

\begin{claim}\label{claim:close-disjoint}
Suppose that $\mathbf x\in\bC_\bo$, $\mathbf y\in\bC_\bp$, and $\ctreedist(\mathbf x,\mathbf y)<L^2+1$.  Then $\fcone(\mathbf x)\neq\fcone(\mathbf y)$.    
\end{claim}

\begin{proofofclaim}{\ref{claim:close-disjoint}}
Letting $\mathbf x,\mathbf y$ be as in the statement, we can choose $x_n,y_n\in C_n$ such that $\omega-\lim_n (x_n,o)=\mathbf x$ and $\omega-\lim_n(y_n,p_n)=\mathbf y$ and $d(x_n,y_n)\leq (L^2+2)r_n$ for $\omega$--a.e. $n$.  The first property implies that $x_n,y_n\in N_n$ for $\omega$--a.e. $n$.  Therefore, our choice of $(p_n)_n$ implies that  $\hat \dist(\hat f(x_n,o),\hat f(y_n,p_n))> \epsilon r_n-2L_0$ for $\omega$--a.e. $n$, and hence $\cdist(\fcone(\mathbf x),\fcone(\mathbf y))\geq \epsilon$.
\end{proofofclaim}

Now let $\mathbf x\in \bC_\bo$ and $\mathbf y\in \bC_\bp$ be arbitrary.  Suppose that $\mathbf z\in\bC_o$ satisfies $\ctreedist(\mathbf y,\mathbf z)\leq 1$.  Then since $\fcone$ is $L_0$--bilipschitz on $\bC_\bo$ and $L_0$--lipschitz on $\tcone$, our choice of $L$ ensures that 
\begin{eqnarray*}
\cdist(\fcone(\mathbf x),\fcone(\mathbf y))&\geq& \cdist(\fcone(\mathbf x),\fcone(\mathbf z))-\frac{L}{100}\\
&\geq& \frac{100(\ctreedist(\mathbf x,\mathbf y)-1)}{L}-\frac{L}{100}.
\end{eqnarray*}
Now, if $\ctreedist(\mathbf x,\mathbf y)\leq L^2+1$, then we are done by Claim \ref{claim:close-disjoint}, so we may assume that $\ctreedist(\mathbf x,\mathbf y)> L^2+1$.  Hence the above estimate shows that $\cdist(\fcone(\mathbf x),\fcone(\mathbf y))\geq 100L-L/100>0$.  Therefore, $\bf\hat f(\bf C_o)$ and $\bf \hat f(C_p)$ are disjoint, as required.
\end{proof}

Fix once and for all a sequence $(p_n)_n$ of points in $T_3$ satisfying the conclusion of Lemma \ref{lem:good_direction}, and let $\bp=\omega-\lim_n p_n$. 

\begin{convention}\label{conv:coordinates}
From here on, we often have to choose bilipschitz homeomorphisms identifying various (sub)spaces with subspaces of $\mathbb R^n$. To avoid cumbersome notation, we will call these any such choice a ``choice of coordinates'', and we will have always chosen coordinates when discussing, for instance, Jacobians. Notice that the property of having full rank Jacobian almost everywhere for a map between spaces identified with subspaces of $\mathbb R^n$ does not depend on the choice of coordinates.
\end{convention}

As a first instance of the convention, we fix a choice of coordinates for $\mathbf F$, identifying it with $\reals^k$.

\begin{defn}[Thickened flat $\bD$]\label{defn:new-block}
Define $\bD=\omega-\lim_n C_n\times[o,p_n]$.  Fix a choice of coordinates identifying $\bD$ with $\reals^k\times[-1,1]$.  
\end{defn}

\begin{defn}[Map $\mathbf h$]\label{defn:new-block-map}
Recall that $\gate\circ \fcone:\tcone\to\hatcone$ is the ultralimit of the maps $\gate_n\circ \hat f$, and define $\mathbf h=\gate\circ\fcone|_{\bD}:\bD\to \mathbf F$.
\end{defn}

Given a choice of coordinates for $\mathbf F$, we can regard $\mathbf h$ as a map $\reals^k\times[-1,1]\to \reals^k$.  In the next lemma, we produce the arc $\alpha$ promised earlier.

 \begin{lem}\label{lem:arc-in-nature}
      There exists $x\in \mathbf F$ such that $\mathbf h^{-1}(x)$ contains a rectifiable arc $\alpha$ with endpoints $a_\bo\in\bC_\bo$ and $a_\bp\in\bC_\bp$, and, moreover, at $\mathcal H^1$-almost-every point of $\alpha$ the Jacobian  $\jacobian{\mathbf h}$ has full rank.
  \end{lem}

 \begin{proof}
 This is because we can apply Corollary \ref{cor:arc} to $\mathbf h$.  Indeed, $\mathbf h$ is lipschitz since $\gate$ and $\fcone$ are.  The map $\gate\circ\fcone:\mathbf C\to \hatcone$ is bilipschitz by our induction hypothesis, and this map is also the restriction to $\mathbf C\subseteq \mathbf D$ of the map $\mathbf h$; we can choose coordinates identifying $\mathbf D$ with $\reals^k\times [-1,1]$ in such a way that $\mathbf C$ is identified with $\reals^k\times \{-1\}$.  Hence the aforementioned corollary applies.
  \end{proof}

From now on, we fix $\alpha$ as in Lemma \ref{lem:arc-in-nature}.  The path $\mathbf h\circ\alpha$ is constant, but $\fcone\circ \alpha$ is not, and therefore has to make progress in some ultralimit of hyperbolic spaces from the HHS structure.  This provides the standard $1$--box $\mathbf F'$ in the statement of the lemma below.

\begin{lemma}
\label{lem:1-box}
    There exists a 1-dimensional standard box $\mathbf F'$, with gate $\mathbf g'$, such that there exists a point $\bf z$ of $\alpha$ where both $\jacobian{\bf h}$ and $\jacobian{\bf h'}$ have full rank, where $\mathbf h':=(\gate'\circ \hat f)|_{\alpha}$.
\end{lemma}

\begin{proof} 
    Lemma \ref{lem:arc-in-nature} provides a rectifiable arc $\alpha$ in $\bD$ and a point $x\in \mathbf F$ such that $\mathbf h(\alpha(t))=x$ for all $t$, and $\alpha$ has one endpoint in $\bC_\bo$ and one in $\bC_\bp$, and $\jacobian{\mathbf h}$ has full rank at $\haus^1$--almost all points in $\alpha$.  We view $\alpha$ as a path $\alpha:[-\ell,\ell]\to \mathbf D$ parametrised by arc length, and identify $\alpha$ with the interval $[-\ell,\ell]$ in the obvious way, with $\alpha(-\ell)\in\bC_\bo$ and $\alpha(\ell)\in\bC_\bp$.

Now apply \cite[Lemma 13.13]{HHS_I} to the points $a_\bo:=\hat f(\alpha(-\ell))$ and $a_\bp:=\hat f(\alpha(\ell))$, to obtain a nontrivial path $\boldgamma$ in $\hatcone$ that is an ultralimit of hierarchy paths $(\gamma_n)_n$ in $\widehat X$, joining points $u_n,v_n$,  with the following properties:
\begin{itemize}
   \item There exists $\bV=\omega-\lim_nV_n$, where each $V_n\in\mathfrak S-\mathfrak S_{ql}$ (because we are working with $\widehat X$, not $X$), and $\mathbf F_\bV=\omega-\lim_n F_{V_n}\subset \hatcone$, such that $\boldgamma\subseteq \mathbf F_\bV$.

    \item If $W_n\propnest V_n$ for $\omega$--a.e. $n$, then $\omega-\lim_n\hat\sigma_{W_n}(u_n,v_n)/r_n=0$.
    
    \item The gate map $\hatcone\to\mathbf F_\bV$ sends $a_\bo$ and $a_\bp$ to distinct endpoints $\mathbf u=(u_n)_n,\mathbf v=(v_n)_n$ of $\boldgamma$. 
\end{itemize}
Since each $\gamma_n$ is a hierarchy path, we have from, e.g. \cite[Lem. 27.1]{CRHK} that $\boldgamma$ is an embedded bilipschitz path, and we therefore conflate $\boldgamma$ with its image. 

Note that $(\gamma_n)_n$ satisfies the properties from Definition \ref{defn:valid-boxes} needed to be a small valid sequence of (rank $1$) hierarchy boxes, so by Lemma \ref{lem:hedge-gate-support}, $\boldgamma$ is a rank--$1$ standard box (see Definition \ref{defn:hedge}), and the support of $\boldgamma$ is $\bV$ by Definition \ref{defn:support-standard-flat}.  Take $\mathbf F'=\boldgamma$.

Letting $\gate':\hatcone\to\mathbf F'$ be the gate map, we thus have $\gate'(a_\bo)\neq \gate'(a_\bp)$.

    We are left to argue that there exists $\bf z\in \alpha$ at which  $\mathbf h$ and $\mathbf h'$ have full rank Jacobian.

Let $A\subseteq\alpha$ be the set of points where $\mathbf h$ has full-rank Jacobian; recall that $\haus^1$--almost all points in $\alpha$ are in $A$.  On the other hand, $\mathbf h':\alpha\to\mathbf F'$ is lipschitz, so by \cite[Thm. 3.3]{Simon:GMT}, there exists a positive $\mathcal H^1$-measure subset $A'$ of $\alpha$ where $\jacobian{\mathbf h'}$ is well-defined and non-zero, and there is therefore some $\mathbf z\in A\cap A'$, as required.
\end{proof}

Thus far, we have a standard $k$--flat $\mathbf F$ with support $\{\bU_1,\ldots,\bU_k\}$ with a map $\mathbf h:\mathbf D\to \mathbf F$, and a standard $1$--box $\mathbf F'$ with support $\{\bU_{k+1}\}$ and a map $\mathbf h':\alpha\to \mathbf F'$ ($\mathbf h$ and $\mathbf h'$ both come from gate maps).

We now show that $\mathbf F,\mathbf F'\hookrightarrow\hatcone$ extend naturally to a bilipschitz embedding $\mathbf F\times \mathbf F'\to\hatcone$ whose image is a standard box.  The following statement also describes the gate map to this product as the product of gate maps to the factors, which then puts us in the situation of Lemma \ref{lem:full_rank_h}.  

\begin{cor}\label{cor:orthogonal-supports}
The standard flats $\mathbf F,\mathbf F'$ have the following properties:
\begin{enumerate}
    \item Any two distinct elements of $\{\bU_1,\ldots,\bU_{k+1}\}$ are orthogonal.\label{item:new-flat-orthogonal}

    \item There is a bilipschitz embedding $\iota:\mathbf F\times \mathbf F'\to\hatcone$ and a point $(f,f')\in\mathbf F\times\mathbf F'$ such that $\iota(a,f')=a$ and $\iota(f,b)=b$ for all $a\in\mathbf F$ and $b\in\mathbf F'$.\label{item:new-flat-product}

    \item Let $\gate'':\hatcone\to\mathbf F''$ be the gate map and let $\mathbf h''=\gate''\circ \fcone|_{\mathbf D}$. Then $\mathbf h''=\iota\circ(\mathbf h|_{\mathbf C_o}\times \mathbf h')$.\label{item:new-flat-gate}

    \item The image of $\iota$ is a standard box (recall Definition \ref{defn:hedge}) $\mathbf F''$ in $\hatcone$.\label{item:new-flat-flat}
\end{enumerate}
\end{cor}

\begin{proof}
The elements of $\{\bU_1,\ldots,\bU_{k}\}$ are pairwise orthogonal, so we need to show that $\bU_{k+1}$ is orthogonal to any $\bU_j$ for $j\leq k$.  

In this proof, let $\pi_{\bU_i}:\mathbf D\to\omega-\lim_n(\mathcal CU_i^n,\dist_{U_i^n}/r_n):=\mathcal C\bU_i$ be the ultralimit of the maps $\pi_{U_i^n}\circ \gate_n\circ \hat f$, and let $\cdist_i=\omega-\lim_n\dist_{U_i^n}/r_n$, so that $(\mathcal C\bU_i,\cdist_i)$ is a real tree.  If $i,j$ are distinct and $k+1\in\{i,j\}$, then below we will consider the possibility that $\bU_i\propnest\bU_j$ (so $U^n_i\propnest U_j^n$ for $\omega$--a.e. $n$) or $\bU_i\transverse\bU_j$ (so $U^n_i\transverse U^n_j$ for $\omega$--a.e. $n$), and we let $\rho^{\bU_i}_{\bU_j}\in\mathcal C\bU_j$ be the point represented by $\left(\rho^{U_i^n}_{U_j^n}\right)_n$.  Note that the consistency axioms for HHSs (\cite[Defn. 1.1]{HHS_II}) pass to rescaled ultralimits; for example, if $\bU_i\transverse\bU_j$, and $\mathbf y\in\mathbf D$, then $\pi_{\bU_j}(\mathbf y)=\rho^{\bU_i}_{\bU_j}$ or the same holds with $i$ and $j$ reversing roles.

Consider the point $\bf z\in\alpha$ as in Lemma \ref{lem:1-box}. We will also choose points $\bf z'$, $\bf w$, $\bf w'$ as follows. First, let $\bf w$ be another point on $\alpha$. Since $\mathbf h(\alpha)=\mathbf x$, we have $\pi_{\bU_j}(\mathbf w)=\pi_{\bU_j}(\mathbf z)$ for $j\leq k$.  Next, since $\jacobian{\mathbf h'}(\mathbf z)$ has rank $1$, and $\jacobian{\mathbf h}$ has full rank almost everywhere along $\alpha$, by Lemma \ref{lem:arc-in-nature}, we can require $\epsilon:=\cdist_{{k+1}}(\mathbf w,\mathbf z)>0$ and also that $\jacobian{\mathbf h}(\mathbf w)$ has full rank. Using the latter property, pick $\mathbf w'\in \mathbf D$ (close to $\mathbf w$) such that $\cdist_{k+1}(\mathbf w,\mathbf w')>0$ but $\cdist_{{k+1}}(\mathbf w,\mathbf w')<\epsilon/10$. Finally, we pick $\mathbf z'$ (close to $\mathbf z$) similarly, and we can arrange $\pi_{\bU_{j}}(\mathbf w')=\pi_{\bU_{j}}(\bf z')$.

The existence of the four points $\bf w$, $\bf w'$, $\bf z$, $\bf z'$ with the stated properties of projections forces $\bU_{k+1}$ to be orthogonal to $\bU_j$, as we now argue by excluding all other cases using the consistency inequalities.

We cannot have $\bU_{k+1}=\bU_j$, since $\pi_{\bU_j}(\mathbf w)=\pi_{\bU_j}(\mathbf z)$ but this does not hold for $\bU_{k+1}$.

Suppose $\bU_{k+1}\transverse \bU_j$. Since
$$\cdist_{k+1}(\mathbf w,\mathbf z)>10\max\{\cdist_{k+1}(\mathbf w',\mathbf w),\cdist_{k+1}(\mathbf z',\mathbf z)\}\ \ \ (*)$$
we have $\rho^{\bU_j}_{\bU_{k+1}}\neq \bf \pi_{\bU_{k+1}}(\bf w), \bf \pi_{\bU_{k+1}}(\bf w')$, or the same for $\bf z, \bf z'$. Assume the former, as the other case is symmetric. Consistency implies $\bf \pi_{\bf U_{j}}(\bf w)=\rho_{\bU_j}^{\bU_{k+1}} =\bf \pi_{\bf U_{j}}(\bf w')$, contradicting $\cdist_{j}(\mathbf w,\mathbf w')>0$.

Suppose $\bU_{k+1}\propnest \bU_j$. Since $\cdist_{k+1}(\mathbf w,\mathbf z)>0$ and $\mathbf \pi_{U_j}(\mathbf w)=\mathbf \pi_{U_j}(\mathbf z)$, by Bounded Geodesic Image (see \cite[Defn. 1.1]{HHS_II}) we must have $\rho_{\bU_j}^{\bU_{k+1}}= \mathbf \pi_{U_j}(\mathbf w)=\mathbf \pi_{U_j}(\mathbf z)$. But then, since $\bf \pi_{\bf U_{j}}(\bf w')=\bf \pi_{\bf U_{j}}(\bf z')$ do not coincide with $\rho_{\bU_j}^{\bU_{k+1}}$, by consistency we should have $\bf \pi_{\bU_{k+1}}(\bf w')=\bf \pi_{\bU_{k+1}}(\bf z')$, contradicting one of the projection properties of the four points.

Suppose $\bf U_{j}\propnest \bU_{k+1}$. Since $\cdist_{k+1}(\mathbf w',\mathbf w)>0$, we must have that $\rho^{\bU_j}_{\bU_{k+1}}$ lies on a geodesic from $\bf \pi_{\bU_{k+1}}(\bf w)$ to $\bf \pi_{\bU_{k+1}}(\bf w')$. The same holds for $\bf z$ and $\bf z'$, but the two relevant geodesics do not intersect because of $(*)$.

\textbf{Conclusion.} Recall that for $i\leq k$, we have a sequence hierarchy intervals $\gamma_n=\prod_i\gamma^n_i\to \bigcap_{i=1}^kP_{U_i^n}$ such that $\omega-\lim_n\gamma_n=\mathbf F$.  Let $a_n,b_n$ be the endpoints of $\gamma_n$.  

Consider a sequence $\alpha_n$ of hierarchy intervals converging to $\mathbf F'$.  Let $E^n$ be the subset of $\bigcap_{i=1}^kP_{U_i^n}$ consisting of those points $a$ such that $\pi_{V}(a)$ is $C$--close to $\pi_{V}(a_n)$ for all $V\nest U_i^n$ and $i\leq k$, where $C$ is a constant that depends on the HHS parameters but not on $n$.  Then the sets $E^n$ are hierarchically quasiconvex subsets such that $\bigcap_{i=1}^kP_{U_i^n}$ is uniformly quasi-isometric to $\prod_{i=1}^kF_{U_i}^{a_n}\times E^n$.  Since $U_{k+1}\orth U_i$ for $i\leq k$, and $\gamma_n\subseteq \prod_{i=1}^kF_{U^n_i}^{a_n}$ for all $n$, we can replace $\alpha_n$ by its image under the gate map to $\gamma_n\times E^n$, and thus assume that the inclusions $\gamma_n,\alpha_n\to \widehat X$ extend to a uniformly quasimedian uniform quasi-isometric embedding $\gamma_n\times \alpha_n\to \widehat X$ whose image is a hierarchy interval.  So, $\mathbf F'':=\omega-\lim_n\gamma_n\times\alpha_n$ is a standard flat supported on $\{\bU_1,\ldots,\bU_{k+1}\}$ and contains $\mathbf F,\mathbf F'$.  Moreover, the gate maps $\gate,\gate'$ restrict on $\mathbf F''$ to maps whose product is a bilipschitz homeomorphism $\mathbf F''\to \mathbf F\times\mathbf F'$, giving items \eqref{item:new-flat-product}, \eqref{item:new-flat-flat}, \eqref{item:new-flat-gate}.
\end{proof}

We fix the notation of Corollary \ref{cor:orthogonal-supports} from now on. Combining Lemma \ref{lem:full_rank_h} and Corollary \ref{cor:orthogonal-supports}.\eqref{item:new-flat-flat} gives:

\begin{cor}
\label{cor:full_rank}
    There exists $\mathbf w\in \mathbf D$ where $\jacobian{\mathbf h''}$ has full rank.
\end{cor}

Now we are ready to complete the proof of the main theorem, by ``zooming in'' on a point of $\mathbf F''$, to pass from having a point of full-rank Jacobian to having a bilipschitz map.

\begin{proof}[Proof of Theorem \ref{thm:main}]
From Corollary \ref{cor:orthogonal-supports}, we have sequences $(U_i^n)_n$ in $\mathfrak S-\mathfrak S_{ql}$, for $1\leq i\leq k+1$, such that for all $i\neq j$ and $\omega$--a.e. $n$, $U^i_n\orth U^j_n$.  Moreover, we have hierarchy intervals $(\beta_n)_n$ such that $\omega-\lim_n\beta_n=\mathbf F''$, the standard box from Corollary \ref{cor:orthogonal-supports}.  Letting $\gate_n:\widehat X\to \beta_n$ be the (coarse) gate, the maps $\gate_n\circ\hat f$ limit to the map $\mathbf h''$ from Corollary \ref{cor:orthogonal-supports}.

Now, recall that in $T_3^k$, we have boxes $C_n=\prod_{i=1}^k[a_i^n,b_i^n]$ and segments $[o,p_n]$ in $T_3$ such that $\omega-\lim_n C_n\times [o,p_n]=\mathbf D$.  

Corollary \ref{cor:full_rank} provides $\mathbf w\in\mathbf D$ such that $\jacobian{\mathbf h''}(\mathbf w)$ has full rank, so there is a constant $L\geq 1$ such that the following holds: for all $\epsilon>0$, there exists $r>0$ such that 
$$\frac{\mathbf d(x,y)}{L}-\epsilon r\leq \cdist(\mathbf h''(x),\mathbf h''(y))\leq L\mathbf d(x,y)+\epsilon r$$ 
for all $x,y\in\mathbf D$ such that $\mathbf d(x,\mathbf w),\mathbf d(y,\mathbf w)<r$.

 Now apply the underspill principle (as in, for instance, the proof of Claim 2 in \cite[Thm. 13.11]{HHS_I}) to obtain a new rescaling sequence $(r'_n)_n$ such that $(e_n)_n\ll(r'_n)_n\ll(r_n)_n$ and the following holds: let $\mathbf d'=\omega-\lim_n d/r'_n$ and let $\hatcone'$ be the asymptotic cone of $\widehat X$ obtained using the new rescaling factor $(r_n')$.  Since $(e_n)_n\ll(r_n')_n$ and $(r_n')_n\ll (s_n)_n$ (using that, by Definition \ref{defn:valid-boxes}, either $(r_n)_n\ll (s_n)_n$ or $s_n=O(r_n)$), the limit $\omega-\lim_n \beta_n$ in $\hatcone'$ is a standard flat, which we denote $\mathbf F'''$, and, redefining $\mathbf h'''$ to be the ultralimit of the same maps $(\gate_n\circ\hat f)_n$ but with the new rescaling, we have
 $$\frac{\mathbf d(x,y)}{L}\leq \cdist(\mathbf h'''(x),\mathbf h'''(y))\leq L\mathbf d(x,y)$$
 for all $x,y\in \omega-\lim_n C_n\times[o,p_n]$.  So, by passing to sub-boxes, there is an ultralimit $\bC$ of boxes in $T^{k+1}_3$ and a standard $(k+1)$--flat $\mathbf F'''$ in the asymptotic cone $\hatcone'$, with gate $\gate'''$, such that the composition $\gate'''\circ\fcone:\mathbf C\to \mathbf F'''$ is bilipschitz.  This completes the proof.
\end{proof}

\bibliographystyle{alpha}
\bibliography{bushy-QR}
\end{document}